\documentclass[10pt,a4paper]{amsart}
\usepackage[margin=1in]{geometry}
\usepackage{amsmath}
\usepackage{amsfonts}
\usepackage{amssymb}
\usepackage{graphicx}
\usepackage{amsthm}
\usepackage{enumitem}
\usepackage{amscd}
\usepackage{mathrsfs}
\usepackage{mathtools}
\usepackage{tikz-cd}
\usepackage{eucal}
\usepackage[colorlinks = true, linkcolor = blue, bookmarks=true]{hyperref}

\usepackage[
backend=biber,
style=alphabetic,
sorting=nyt,
maxbibnames=99,
abbreviate=false, 
isbn=false,
]{biblatex}
\author{Zhiyuan Chen}
\address{Department of Mathematics, Princeton University, Princeton, NJ, 08544-1000, USA}
\email{zc5426@princeton.edu}

\title[Finite generation and flatness]{Finite generation of higher rank quasi-monomial valuations via the extended Rees algebra}

\theoremstyle{plain}
\newtheorem{thm}{Theorem}[section]
\newtheorem*{thm*}{Theorem}

\newtheorem{lem}[thm]{Lemma}

\newtheorem{cor}[thm]{Corollary}

\theoremstyle{definition}
\newtheorem{defn}[thm]{Definition}

\newtheorem{exmp}[thm]{Example}
\newtheorem{rem}[thm]{Remark}

\let\originalleft\left
\let\originalright\right
\renewcommand{\left}{\mathopen{}\mathclose\bgroup\originalleft}
\renewcommand{\right}{\aftergroup\egroup\originalright}

\renewcommand{\subset}{\subseteq}

\begin{document}

\begin{abstract}
    In the algebraic theory of K-stability, one of the most challenging problems is to show the graded algebra associated with certain higher rank quasi-monomial valuations are finitely generated. In the global case of Fano varieties (\cite{LXZ_finite_generation}) and local case of klt singularities (\cite{XZ_stable_deg}), the finite generation has been proved for quasi-monomial valuations on models of qdlt Fano type, by the multi-step degeneration method. In this paper, we generalize these results using a different argument, by studying the extended Rees algebra via a more algebraic approach. As consequences, our results apply to fibrations of Fano type with singularities worse than qdlt, and graded algebras coming from the multi-section ring of arbitrary divisors. 
\end{abstract}

\maketitle

\tableofcontents

\section{Introduction}

Given a klt Fano pair $(X, \Delta)$ over a field $\Bbbk$ of characteristic zero, the algebraic theory of K-stability studies various invariants defined for real valuations $v$ on $X$ and the induced filtrations $\mathcal{F}_v$ on the section ring $R = R(X, -(K_X + \Delta))$. For example, we have the $\delta$-invariant $\delta(v) = \frac{A_{X,\Delta}(v)}{S(\mathcal{F}_v)}$, and the Fujita--Li criterion says $\inf_v \delta(v) \geq 1$ if and only if $(X, \Delta)$ is K-semistable. The infimum is in fact attained, and it is very important to understand the geometry of minimizing valuations for the $\delta$-invariant. In the case $\inf_v \delta(v) \leq 1$, a minimizing valuation $v$ induces an optimal destabilization $(X_0, \Delta_0)$ of $(X, \Delta)$, given by $X_0 = \mathop{\mathrm{Proj}}(\mathrm{gr}_v(R))$, where $\mathrm{gr}_v(R)$ is associated graded algebra of the filtration $\mathcal{F}_v$ on $R$. Here the most crucial and difficult step is to show that $\mathrm{gr}_v(R)$ is a finitely generated $\Bbbk$-algebra. The finite generation of $\mathrm{gr}_v(R)$ also completes the proofs of many other big theorems in the K-stability theory of Fano varieties, including the equivalence of K-stability and uniform K-stability, and the properness and projectivity of the K-moduli space. See \cite{LXZ_finite_generation} and \cite{Xu_K-Stability_Book} for the aforementioned results. 

In parallel with the global theory, the finite generation problem also appears in the local K-stability theory for klt singularities. For a klt singularity $x \in (X = \mathop{\mathrm{Spec}}(A), \Delta)$, the minimizing valuation $v$ for the normalized volume $\widehat{\mathrm{vol}}_{X,\Delta}$ induces a degeneration of $(X, \Delta)$ to a K-semistable Fano cone $(X_0, \Delta_0)$, where $X_0 = \mathop{\mathrm{Spec}}(\mathrm{gr}_v(A))$. Similar to the global case, the crucial step is to show that $\mathrm{gr}_v(A)$ is a finitely generated algebra; see \cite{XZ_stable_deg} and \cite{Zhuang_survey}. More recently, a similar finite generation problem for Fano fibrations is asked in \cite[Conj.\ 6.4]{SZ_KR_shrinker}. 

\subsection{The finite generation problem} The strategy to prove finite generation for a minimizing valuation $v$ can be roughly described as three steps: 
\begin{enumerate}[label=(\arabic*), nosep]
    \item Show that $v$ is quasi-monomial, and is an lc place of a $\mathbb{Q}$-complement. 
    \item Show that in addition $v$ satisfies a ``special'' geometric condition among all lc places. 
    \item Show that the ``special'' geometric condition implies finite generation. 
\end{enumerate}
The first step is settled by \cite{Xu_qm} (see also \cite[Thm.\ A.2]{BLX_openness}), and provides a natural setup to study the finite generation problem using techniques from the Minimal Model Program (MMP). Indeed, if $v$ is moreover divisorial, then the graded ring $\mathrm{gr}_v(R)$ is finitely generated by \cite[Cor.\ 1.4.3]{BCHM}. But, if $v$ has rational rank $>1$, this condition doesn't imply finite generation; see \cite[Thm.\ 6.1]{LXZ_finite_generation}. 

A special condition that works for (2) and (3) was first found in \cite{LXZ_finite_generation}. It is relatively easy to get (2), but the proof of (3) is convoluted, and uses certain boundedness results for Fano varieties. Later, a different approach for the finite generation is developed in \cite{Xu_towards} and completed in \cite{XZ_stable_deg}, under the same special condition, but avoids the boundedness and works in both the local case for klt singularities and the global case for Fano varieties. The special condition is that $v$ is a \emph{monomial lc place of a special $\mathbb{Q}$-complement} with respect to some log resolution, and can be reduced to the following using the MMP: There is a birational model $f \colon Y \to (X, \Delta)$, and an effective $\mathbb{Q}$-divisor $D \geq f^{-1}_{*}\Delta$ on $Y$ such that $(Y, D)$ is dlt, $-(K_Y + D)$ is ample, and $v \in \mathrm{QM}^{\circ}(Y, E)$ for $E = \lfloor D \rfloor$; we call $(Y, E)$ a model of dlt Fano type. 

In this paper, we focus on (3), and generalize the finite generation results in \cite{LXZ_finite_generation} and \cite{XZ_stable_deg} via a new approach that is of more algebraic nature. The main result is the following: 

\begin{thm}[Theorem \ref{finite_generation_thm}] \label{main_thm_1}
    Let $\Bbbk$ be a field of characteristic $0$, and $S = \mathop{\mathrm{Spec}}(A)$ be an affine scheme essentially of finite type over $\Bbbk$. Let $(Y, D)$ be an lc pair, and $g \colon Y \to S$ be a proper morphism such that $-(K_Y + D)$ is $g$-ample. Suppose $E_1, \ldots, E_r$ are reduced divisors on $Y$ with $E = \sum_{i = 1}^{r} E_i = \lfloor D \rfloor$, such that every lc center of $(Y, D)$ is contained in $\mathrm{Supp}(E)$, and each $E_i$ is $\mathbb{Q}$-Cartier at the lc centers\footnote{When making this assumption, we always exclude the generic points of $Y$, though they are sometimes also considered to be lc centers. }. Let $L_1, \ldots, L_s$ be $\mathbb{Q}$-Cartier $\mathbb{Z}$-divisors on $Y$, and
    \begin{equation*}
        R = R(L_1, \ldots, L_s) = \bigoplus_{n = (n_1, \ldots, n_s) \in \mathbb{N}^s}H^0(Y, \mathscr{O}_Y(n_1L_1 + \cdots + n_sL_s))
    \end{equation*}
    be the multi-section ring. Assume that $Z = \bigcap_{i=1}^{r} E_i$ is irreducible with the generic point $\zeta \in Z$. Suppose $v \in \mathrm{QM}^{\circ}_{\zeta}(Y, E)$ (see Definition \ref{def_simple_toroidal_QM}), and $\mathcal{F}_{v}$ is the filtration on $R$ induced by $v$, with the associated graded algebra $\mathrm{gr}_{v}(R)$. Then $\mathrm{gr}_{v}(R)$ is of finite type over $A$. 
\end{thm}

\begin{rem}
    (1) The ring $R$ itself is finitely generated by \cite[Cor.\ 1.1.9]{BCHM}. 

    (2) If $(X, \Delta)$ is a klt Fano pair over $\Bbbk$ and $f \colon (Y, E) \to (X, \Delta)$ is a model of dlt Fano type, then our theorem applies for $S = \mathop{\mathrm{Spec}}(\Bbbk)$ and $L = f^{*}(-m(K_X + \Delta))$, where $m \in \mathbb{Z}_{>0}$ such that $-m(K_X + \Delta)$ is Cartier, and recovers the finite generation in \cite{LXZ_finite_generation}. If $x \in (X = \mathop{\mathrm{Spec}}(A), \Delta)$ is a klt singularity and $f \colon (Y, E) \to (X, \Delta)$ is a model of dlt Fano type, then our theorem for $S = X$ and $L = 0$ recovers \cite{XZ_stable_deg}. 

    (3) In comparison with \cite{XZ_stable_deg}, where $X$ is of primary interest and $Y$ is adapted to the valuation $v$, we consider the section rings of divisors on $Y$ directly. This viewpoint allows us to deal with arbitrary divisors and their multi-section ring. 
\end{rem}

Our Theorem \ref{main_thm_1} applies to more valuations than the special valuations in \cite{LXZ_finite_generation} and \cite{XZ_stable_deg}. For a klt singularity $x \in (X, \Delta)$, if $v = \mathrm{ord}_E$ is a divisorial valuation over $X$ centered at $x$ and is an lc place of a $\mathbb{Q}$-complement, then there exists a projective birational morphism $f \colon Y \to X$ such that $E \subset Y$ is the unique exceptional divisor and $-E$ is $f$-ample by \cite[Cor.\ 1.4.3]{BCHM}. Then $(Y, f^{-1}_{*}\Delta + E)$ satisfies the conditions in Theorem \ref{main_thm_1} (the finite generation in this case is already known by \cite{BCHM}), but $v = \mathrm{ord}_E$ satisfies the special condition in \cite{XZ_stable_deg} if and only if $E$ is a \emph{Koll\'ar component}, that is, the pair $(Y, f^{-1}_{*}\Delta + E)$ is plt. 

In the global case, we also give an example of a valuation of rational rank $2$ on a Fano variety which satisfies conditions in Theorem \ref{main_thm_1}, but not the special condition in \cite{LXZ_finite_generation}. See Example \ref{example_weakly_special}.

\subsection{The extended Rees algebra} Let $(X = \mathop{\mathrm{Spec}}(A), \Delta)$ is a klt pair, and $f \colon (Y, E = \sum_{i=1}^{r} E_i) \to X$ be a birational model extracting lc places $E_i$ of a $\mathbb{Q}$-complement. A key observation from \cite{Xu_towards} is that if $\mathrm{gr}_v(A)$ is finitely generated for $v \in \mathrm{QM}^{\circ}(Y, E)$, then $X_0 = \mathop{\mathrm{Spec}}(\mathrm{gr}_v(A))$ is a multi-step degeneration of $X$ induced by the divisors $E_1, \ldots, E_r$; conversely, if such a degeneration $X_0$ is integral, then $\mathrm{gr}_v(A)$ is finitely generated $X_0 = \mathop{\mathrm{Spec}}(\mathrm{gr}_v(A))$. If $(Y, E)$ is a model of dlt Fano type, it is proved in \cite{XZ_stable_deg} that the total family $\mathcal{X} \to \mathbb{A}^r$ of the multi-step degeneration is a family of klt pairs, hence the central fiber $X_0$ is integral. 

Algebraically, the family is $\mathcal{X} = \mathop{\mathrm{Spec}}(\mathcal{R})$, where $\mathcal{R}$ is the extended Rees algebra for the filtrations on $A$ induced by $E_i$, which is naturally an algebra over $\Bbbk[t_1, \ldots, t_r]$. Under the conditions in Theorem \ref{main_thm_1}, we consider the extended Rees algebra $\mathcal{R}$ for the filtrations on $R$ induced by $E_i$. We show that if $\mathcal{R}$ is flat over $\Bbbk[t_1, \ldots, t_r]$ and $\overline{\mathcal{R}} = \mathcal{R}/(t_1, \ldots, t_r)\mathcal{R}$ is an integral domain, then $\mathrm{gr}_v(R) \simeq \overline{\mathcal{R}}$. Note that $\mathcal{R}$ is finitely generated by \cite{BCHM}, hence we deduce Theorem \ref{main_thm_1} from the following: 

\begin{thm}[Theorem \ref{the_central_fiber_of_Rees_alg}] \label{main_thm_2}
    Let $\Bbbk$ be a field of characteristic $0$, and $S = \mathop{\mathrm{Spec}}(A)$ be an affine scheme essentially of finite type over $\Bbbk$. Let $(Y, \Delta)$ be an lc pair, and $g \colon Y \to S$ be a proper morphism such that $-(K_Y + \Delta)$ is $g$-ample. Suppose $E_1, \ldots, E_r$ are reduced divisors on $Y$ with $E = \sum_{i = 1}^{r} E_i = \lfloor \Delta \rfloor$, such that every lc center of $(Y, \Delta)$ is contained in $\mathrm{Supp}(E)$, and each $E_i$ is $\mathbb{Q}$-Cartier at the lc centers. For each $m = (m_1, \ldots, m_r) \in \mathbb{Z}^r$, write 
    \begin{equation*}
        E(m) \coloneqq \sum_{i=1}^{r} \max(m_i, 0) E_i. 
    \end{equation*}
    Let $L_1, \ldots, L_s$ be $\mathbb{Q}$-Cartier $\mathbb{Z}$-divisors on $Y$, 
    \begin{equation*}
        R = R(L_1, \ldots, L_s) = \bigoplus_{n = (n_1, \ldots, n_s) \in \mathbb{N}^s}H^0(Y, \mathscr{O}_Y(n_1L_1 + \cdots + n_sL_s))
    \end{equation*}
    be the multi-section ring, and 
    \begin{equation*}
        \mathcal{R} = \bigoplus_{n \in \mathbb{N}^s} \bigoplus_{m \in \mathbb{Z}^r} H^0(Y, \mathscr{O}_Y(n_1L_1 + \cdots + n_sL_s - E(m))) t_1^{-m_1} \cdots t_r^{-m_r} \subset R[t_1^{\pm 1}, \ldots, t_r^{\pm r}]
    \end{equation*}
    be the extended Rees algebra. Then the following hold: 
    \begin{enumerate}[label=\emph{(\arabic*)}, nosep]
        \item $\mathcal{R}$ is of finite type over $A$. 
        \item $\mathcal{R}$ is flat over $\Bbbk[t_1, \ldots, t_r]$. 
        \item Let $Z = \bigcap_{i = 1}^{r} E_i$, with the total ring of fractions $K(Z)$. Then there is an injective ring map 
        \begin{equation*}
            \mathcal{R}/(t_1, \ldots, t_r)\mathcal{R} \hookrightarrow K(Z)[\mathbb{N}^s \times \mathbb{N}^r]. 
        \end{equation*}
        In particular, if $Z$ is irreducible, then $\mathcal{R}/(t_1, \ldots, t_r)\mathcal{R}$ is an integral domain. 
    \end{enumerate}
\end{thm}

In the multi-step degeneration of \cite{Xu_towards} and \cite{XZ_stable_deg}, the flatness of $\mathcal{R}$ over $\Bbbk[t_1, \ldots, t_r]$ is implicit, and only follows from the conclusion that $\mathcal{X} \to \mathbb{A}^r$ is a family of klt pairs. Instead, we will take a more algebraic approach to the flatness. This allows us to consider general divisors $L_1, \ldots, L_s$ on $Y$ and their multi-section ring, for which the geometry of $\mathop{\mathrm{Spec}}(\mathcal{R}) \to \mathbb{A}^r$ is more complicated.

\begin{proof}[Sketch of the proof of Theorem \ref{main_thm_2}]
    By the local criterion of flatness (and the grading), we need to show 
\begin{equation*}
    \mathrm{Tor}_1^{\Bbbk[t_1,\ldots,t_r]}(\Bbbk, \mathcal{R}) = 0. 
\end{equation*}
Then we can use the Koszul complex to resolve $\Bbbk = \Bbbk[t_1, \ldots, t_r]/(t_1, \ldots, t_r)$, so it suffices to show 
\begin{equation*}
    H_p(K_{\bullet}(\mathcal{R}; t_1, \ldots, t_r)) = 0
\end{equation*}
for all $p > 0$. Note that multiplication by $t_i$ on $\mathcal{R}$ corresponding to the inclusion 
\begin{equation*}
    H^0(Y, \mathscr{O}_Y(L - E_i)) \to H^0(Y, \mathscr{O}_Y(L))
\end{equation*}
for each $L = n_1L_1 + \cdots + n_sL_s - E(m)$. Thus, the summand of $K_{\bullet}(\mathcal{R}; t_1, \ldots, t_r)$ in each degree is a complex $C_{\bullet}$ of the form 
\begin{equation*}
    0 \to H^0(Y, \mathscr{O}_Y(L - E_1 - \cdots - E_r)) \to \cdots \to \bigoplus_{i=1}^{r} H^0(Y, \mathscr{O}_Y(L - E_i)) \to H^0(Y, \mathscr{O}_Y(L)) \to 0. 
\end{equation*}
We can form a similar complex $\mathscr{C}_{\bullet}$ of $\mathscr{O}_Y$-modules of the form 
\begin{equation*}
    0 \to \mathscr{O}_Y(L - E_1 - \cdots - E_r) \to \cdots \to \bigoplus_{i=1}^{r} \mathscr{O}_Y(L - E_i) \to \mathscr{O}_Y(L) \to 0, 
\end{equation*}
so that $C_p = H^0(Y, \mathscr{C}_p)$ for each $p$. 

We may assume each $E_i$ is $\mathbb{Q}$-Cartier by a small birational modification of $Y$ that does not affect $\mathcal{R}$. By the last paragraph, we need to consider the complexes $\mathscr{C}_{\bullet}$ and $C_{\bullet}$ for a $\mathbb{Q}$-Cartier $\mathbb{Z}$-divisor $L$. First, we show that $H_p(\mathscr{C}_{\bullet}) = 0$ for all $p > 0$. In fact, after passing to a finite abelian cover, each $E_i$ is Cartier, and their defining equations form a regular sequence; see Lemma \ref{Koszul_complex_div}. If all the cohomology $H^q(Y, \mathscr{C}_p)$ vanish for $q > 0$, then we can conclude that $H_p(C_{\bullet}) = 0$ for all $p > 0$ by a spectral sequence. However, one shouldn't expect the vanishing holds on $Y$. Instead, we show that if $L$ is pseudo-effective (otherwise $C_{\bullet} = 0$ is trivially exact), then we can replace $Y$ with a semi-ample model $Y'$ for $L$ using the MMP, and the vanishing theorems on $Y'$ imply $H_p(C_{\bullet}) = 0$ for all $p > 0$; see Lemma \ref{pass_to_sa_model} and Lemma \ref{Koszul_exact_on_sections}. 

Moreover, the central fiber $\mathcal{R}/(t_1, \ldots, t_r)\mathcal{R}$ is a direct sum of modules 
\begin{equation*}
    H^0(Y, \mathscr{O}_Y(L)) \bigg/ \sum_{i=1}^{r} H^0(Y, \mathscr{O}_Y(L - E_i)). 
\end{equation*}
For simplicity, assume that $Z$ is irreducible here. We show that this module is nonzero only if the MMP $Y \dashrightarrow Y'$ is a local isomorphism at the generic point $\zeta$ on $Z$, and it is isomorphic to $H^0(Z', \mathscr{L}')$ for a torsion-free module $\mathscr{L}'$ of rank $1$ on the image $Z'$ of $Z$. Hence we get an inclusion from $\mathcal{R}/(t_1, \ldots, t_r)\mathcal{R}$ to a polynomial ring over $K(Z)$. 
\end{proof}

\subsection{The structure of the graded ring} Combining Theorem \ref{main_thm_1} and Theorem \ref{main_thm_2}, we can relate the graded ring $\mathrm{gr}_v(R) \simeq \mathcal{R}/(t_1,\ldots,t_r)\mathcal{R}$ with a polynomial ring over the function field $K(Z)$ of the center $Z = \bigcap_{i=1}^{r} E_i$. Indeed, in the sketch of proof above, the nonzero graded pieces $\mathcal{R}/(t_1,\ldots,t_r)\mathcal{R}$ are given by sections of certain sheaves on birational models of $Z$. Keeping track of the proof, we have a more explicit result for the structure of $\mathrm{gr}_v(R)$ as follows. 

\begin{cor}[Corollary \ref{graded_ring_from_polyhedral_divisorial_sheaf}]
    Keep the assumptions in Theorem \ref{main_thm_2}. Assume that $Z = \bigcap_{i=1}^{r} E_i$ is irreducible with the generic point $\zeta \in Z$. Write 
    \begin{equation*}
        \overline{\mathcal{R}} = \mathcal{R}/(t_1, \ldots, t_r)\mathcal{R} = \bigoplus_{u \in \mathbb{N}^{s+r}} \overline{\mathcal{R}}_u, 
    \end{equation*}
    with the grading by $\mathbb{N}^{s+r}$ induced by the natural grading on $\mathcal{R}$. Let $\sigma \subset \mathbb{R}^{s+r}$ denote the closed convex cone spanned by all $u \in \mathbb{N}^{s+r}$ such that $\overline{\mathcal{R}}_u \neq 0$. Then $\sigma$ is a convex rational polyhedral cone, and for $u = (n,m) \in \mathbb{N}^{s+r}$, $u \in \sigma$ if and only if $\zeta \notin \mathbf{B}(n_1L_1 + \cdots + n_sL_s - m_1E_1 - \cdots - m_rE_r)$. 

    Assume furthermore that $(Y, \Delta)$ is dlt. Then there exists a positive integer $\ell$, a projective birational morphism $\rho \colon W \to Z$ with $W$ regular, and a collection of invertible $\mathscr{O}_W$-modules $\{\mathscr{L}(u)\}_{u \in \sigma \cap \ell \mathbb{N}^{s+r}}$ with nonzero maps 
    \begin{equation*}
        \mu_{u,u'} \colon \mathscr{L}(u) \otimes \mathscr{L}(u') \to \mathscr{L}(u+u')
    \end{equation*}
    such that the following hold: 
    \begin{enumerate}[label=\emph{(\arabic*)}, nosep]
        \item $\bigoplus_{u} \mathscr{L}(u)$ form a commutative $\mathscr{O}_W$-algebra, 
        \item there is a canonical isomorphism of graded rings 
        \begin{equation*}
            \overline{\mathcal{R}}_{(\ell)} \simeq \bigoplus_{u \in \sigma \cap \ell \mathbb{N}^{s+r}} H^0(W, \mathscr{L}(u))
        \end{equation*}
        where $\overline{\mathcal{R}}_{(\ell)} = \bigoplus_{u \in \ell\mathbb{N}^{s+r}} \overline{\mathcal{R}}_u \subset \overline{\mathcal{R}}$ is the $\ell$-th Veronese subring. 
        \item there is a decomposition $\sigma = \bigcup_{\lambda \in \Lambda} \sigma_{\lambda}$ of $\sigma$ into finitely many convex rational polyhedral cones $\sigma_{\lambda}$, such that if $u, u' \in \sigma_{\lambda} \cap \ell \mathbb{N}^{s+r}$ for some $\lambda \in \Lambda$, then the map $\mu_{u,u'}$ is an isomorphism. 
    \end{enumerate}
\end{cor}

Suppose $x \in (X = \mathop{\mathrm{Spec}}(A), \Delta)$ is a klt singularity, and $v$ is a quasi-monomial valuation as in Theorem \ref{main_thm_1} on a birational model $f \colon (Y, E) \to X$. Then one usually consider $X_0 = \mathop{\mathrm{Spec}}(\mathrm{gr}_v(A))$ with the natural torus action given by the $\mathbb{N}^r$-grading, which is important in the local K-stability theory. By the Corollary above, the weight cone of $X_0$ is spanned by all $m \in \mathbb{N}^r$ such that $\zeta \notin \mathbf{B}(-m_1E_1 - \cdots - m_rE_r)$. It seems to the author that this cannot be deduced directly from the multi-step degeneration in \cite{XZ_stable_deg}. 

\begin{rem}
    The results from higher-dimensional geometry that we need are the standard Kawamata--Viehweg vanishing theorem and the MMP as in \cite{BCHM}. Both are known to be true in more general settings, see \cite{Murayama_vanishing} and \cite{LM_MMP_excellent}. Therefore, all of our results also hold if we assume $A$ is an excellent $\Bbbk$-algebra with a dualizing complex, by the same proof (but with more care if $A$ is not a local ring when we apply Bertini theorems; see \cite[\S 10]{LM_MMP_excellent}). 
\end{rem}

\subsection*{Acknowledgment} I would like to thank my advisor Chenyang Xu for his encouragement and support on this research. I would also like to thank Junyao Peng, Harold Blum, and Ziquan Zhuang for helpful discussions and comments. The author is partially supported by the Simons Collaboration on Moduli of Varieties and the NSF grant DMS-2201349. 

\section{Preliminaries}

\subsection{Birational geometry}

Let $\Bbbk$ be a field of characteristic zero. All schemes we consider are quasi-compact, separated, and essentially of finite type over $\Bbbk$. 

\begin{defn}
    A pair $(X, \Delta)$ consists of a normal equidimensional scheme $X$ and an effective $\mathbb{Q}$-divisor $\Delta$ on $X$. Let $K_X$ denote a \emph{canonical divisor} of $X$. 

    We follow the standard terminology for singularities of pairs as in \cite[Def.\ 2.8]{Kollar_singularity}. We always use \emph{klt} short for ``Kawamata log terminal'', \emph{dlt} for ``divisorial log terminal'', and \emph{lc} for ``log canonical''. 

    A normal scheme $X$ is said to be of \emph{klt type} if there exists a Zariski open covering $X = \bigcup_{i} X_i$ and effective $\mathbb{Q}$-Cartier divisors $\Delta_i$ on $X_i$ such that each $(X_i, \Delta_i)$ is klt. 

    Suppose $(X, \Delta)$ is an lc pair. A point $z \in X$ is called an lc center of $(X, \Delta)$ if there is a prime divisor $E$ over $X$ with center $z$ such that $A_{X,\Delta}(E) = 0$; see \cite[Def.\ 4.15]{Kollar_singularity}. 
\end{defn}

\begin{defn}
    A rational map $\pi \colon Y \dashrightarrow X$ between normal schemes is called a \emph{(proper) birational contraction} if there exists a diagram 
    \begin{equation*}
        \begin{tikzcd}
            & W \ar[ld, "q"'] \ar[rd, "p"] & \\
            Y \ar[rr, "\pi", dashed] & & X
        \end{tikzcd}
    \end{equation*}
    where $W$ is a normal scheme, $p$ and $q$ are proper birational morphisms with $\pi = p \circ q^{-1}$, such that $\pi^{-1}$ has no exceptional divisors, that is, every $q$-exceptional divisor is also $p$-exceptional. 

    Let $L$ be a $\mathbb{Q}$-Cartier $\mathbb{Q}$-divisor on $Y$ such that $M = \pi_{*}L$ is also $\mathbb{Q}$-Cartier. The birational contraction $\pi$ is said to be \emph{$L$-non-positive} if $q^{*}L = p^{*}M + E$ for some $p$-exceptional effective $\mathbb{Q}$-divisor $E$ on $W$. 
\end{defn}

\begin{lem} \label{compare_sections}
    Let 
    \begin{equation*}
        \begin{tikzcd}
            Y \ar[rr, "\phi", dashed] \ar[rd, "g"'] & & Y' \ar[ld, "g'"] \\
            & S &
        \end{tikzcd}
    \end{equation*}
    be a commutative diagram of schemes, such that $Y$ and $Y'$ are normal, $g$ and $g'$ are proper, and $\phi$ is a birational contraction. Suppose $L$ is a $\mathbb{Q}$-Cartier $\mathbb{Z}$-divisors on $Y$ such that $L' = \phi_{*}L$ is also $\mathbb{Q}$-Cartier, and $\phi$ is $L$-non-positive. Let $E$ be an effective $\mathbb{Z}$-divisor on $Y$ with no $\phi$-exceptional components. Then 
    \begin{equation*}
        g_{*}\mathscr{O}_Y(L - E) \simeq g'_{*}\mathscr{O}_{Y'}(L' - E'). 
    \end{equation*}
    where $E' = \phi_{*}E$. 
\end{lem}

\begin{proof}
    Let $W$ be a normal scheme with proper birational morphisms $q \colon W \to Y$ and $q' \colon W \to Y'$ such that $\phi = q' \circ q^{-1}$. By definition, that $\phi$ is $L$-non-positive means 
    \begin{equation*}
        q^{*}L = q'^{*}L' + F
    \end{equation*}
    for some $q'$-exceptional effective $\mathbb{Q}$-divisor $F$ on $W$. Let $L_W = \lfloor q^{*}L \rfloor$ and $E_W = q^{-1}_{*}E = q'^{-1}_{*}E'$. Then 
    \begin{equation*}
        -(L_W - E_W) \sim_{Y, \mathbb{Q}} E_W + (q^{*}L - \lfloor q^{*}L \rfloor)
    \end{equation*}
    where $(q^{*}L - \lfloor q^{*}L \rfloor)$ is $q$-exceptional, $\lfloor (q^{*}L - \lfloor q^{*}L \rfloor) \rfloor = 0$, and $E_W$ is effective without $q$-exceptional components. By \cite[Lem.\ 7.30]{Kollar_singularity}, we have 
    \begin{equation*}
        q_{*}\mathscr{O}_{W}(L_W - E_W) \simeq \mathscr{O}_{Y}(q_{*}(L_W - E_W)) = \mathscr{O}_{Y}(L - E). 
    \end{equation*}
    Note that $B_W \coloneqq L_W - \lfloor q'^{*}L' \rfloor$ is effective and $q'$-exceptional, and 
    \begin{equation*}
        -(\lfloor q'^{*}L' \rfloor - E_W) \sim_{Y', \mathbb{Q}} E_W + (q'^{*}L' - \lfloor q'^{*}L' \rfloor). 
    \end{equation*}
    Thus, by \cite[Lem.\ 7.30]{Kollar_singularity}, we have 
    \begin{equation*}
        q'_{*}\mathscr{O}_W(L_W - E_W) = q'_{*}\mathscr{O}_W(B_W + (\lfloor q'^{*}L' \rfloor - E_W)) \simeq \mathscr{O}_{Y'}(q'_{*}(\lfloor q'^{*}L' \rfloor - E_W)) = \mathscr{O}_{Y'}(L' - E'). 
    \end{equation*}
    Hence 
    \begin{equation*}
        g_{*}\mathscr{O}_{Y}(L-E) \simeq g_{*}q_{*}\mathscr{O}_W(L_W - E_W) = g'_{*}q'_{*}\mathscr{O}_W(L_W - E_W) \simeq g'_{*}\mathscr{O}_{Y'}(L' - E')
    \end{equation*}
    since $g \circ q = g' \circ q'$. 
\end{proof}

\subsection{Filtrations and extended Rees modules}

\begin{defn} \label{the_complex_C}
    Let $\mathsf{A}$ be an abelian category, and $M \in \mathsf{A}$. For a set of subobjects $M_1, \ldots, M_r \subset M$, let $C_{\bullet}(M; (M_i)_{i=1,\ldots,r})$ denotes the following chain complex: 
    \begin{equation*}
        \begin{aligned}
            0 \to M_1 \cap \cdots \cap M_r \to &\cdots \to \bigoplus_{1 \leq i_1 < \cdots < i_k \leq r} M_{i_1} \cap \cdots \cap M_{i_k} \to \\
            &\cdots \to \bigoplus_{1 \leq i < j \leq r} M_i \cap M_j \to \bigoplus_{i=1}^{r} M_i \to M \to 0
        \end{aligned}
    \end{equation*}
    where $M$ is placed in degree $0$, and the map 
    \begin{equation*}
        M_{i_1} \cap \cdots \cap M_{i_k} \to M_{i_1} \cap \cdots \cap \widehat{M}_{i_p} \cap \cdots \cap M_{i_k}
    \end{equation*}
    is given by $(-1)^{p+1}$ times the canonical inclusion. 
\end{defn}

\begin{defn}
    Let $R$ be a ring, and $M$ be an $R$-module. Suppose $\mathcal{F}_1, \ldots, \mathcal{F}_r$ are decreasing $\mathbb{Z}$-valued filtrations on $M$. The \emph{extended Rees module} is 
    \begin{equation*}
        \mathrm{Rees}_{\mathcal{F}_1, \ldots, \mathcal{F}_r}(M) \coloneqq \bigoplus_{(m_1, \ldots, m_r) \in \mathbb{Z}^r} (\mathcal{F}_1^{m_1}M \cap \cdots \cap \mathcal{F}_r^{m_r}M) t_1^{-m_1} \cdots t_r^{-m_r}
    \end{equation*}
    which is an $R[t_1, \ldots, t_r]$-submodule of $M[t_1^{\pm 1}, \ldots, t_r^{\pm 1}]$. 
\end{defn}

\begin{lem} \label{flatness_criterion_of_Rees}
    Let $\Bbbk$ be a field, $R$ be a $\Bbbk$-algebra, and $M$ be an $R$-module. Suppose $\mathcal{F}_1, \ldots, \mathcal{F}_r$ are $\mathbb{Z}$-valued decreasing filtrations on $M$. Then the following are equivalent: 
    \begin{enumerate}[label=\emph{(\arabic*)}, nosep]
        \item The extended Rees module $\mathrm{Rees}_{\mathcal{F}_1, \ldots, \mathcal{F}_r}(M)$ is flat over $\Bbbk[t_1, \ldots, t_r]$; 
        \item For all $m = (m_1, \ldots, m_r) \in \mathbb{Z}$ and $I \subset \{1, \ldots, r\}$, 
        \begin{equation*}
            H_p\left( C_{\bullet}(\mathcal{F}_1^{m_1}M \cap \cdots \cap \mathcal{F}_r^{m_r}M; (\mathcal{F}_1^{m_1}M \cap \cdots \cap \mathcal{F}^{m_i+1}M \cap \cdots \cap \mathcal{F}_r^{m_r}M)_{i \in I}) \right) = 0
        \end{equation*}
        for all $p > 0$. 
    \end{enumerate}
\end{lem}

\begin{proof}
    Write $P = \Bbbk[t_1, \ldots, t_r]$ and $F = \mathrm{Rees}_{\mathcal{F}_1, \ldots, \mathcal{F}_r}(M)$, which are graded by $\mathbb{Z}^r$. Then $F$ is flat over $P$ if and only if 
    \begin{equation*}
        \mathrm{Tor}_p^P(P/\mathfrak{a}, F) = 0
    \end{equation*}
    for all $\mathbb{Z}^r$-graded ideals $\mathfrak{a} \subset P$ and all $p > 0$. Moreover, it suffices to consider graded prime ideals $\mathfrak{a} \subset P$. Then $\mathfrak{a} = (t_{i_1}, \ldots, t_{i_s})$ for some $I = \{i_1, \ldots, i_s\} \subset \{1, \ldots, r\}$. Now $P/\mathfrak{a}$ has a finite free resolution given by the Koszul complex $K_{\bullet}(P; (t_i)_{i \in I})$. Hence the flatness of $F$ is equivalent to that 
    \begin{equation*}
        H_p(K_{\bullet}(P; (t_i)_{i \in I}) \otimes_P F) = 0
    \end{equation*}
    for all $p > 0$ and $I \subset \{1, \ldots, r\}$. 

    To simplify the notations, we may assume that $I = \{1, \ldots, s\}$. The Koszul complex 
    \begin{equation*}
        K_{\bullet}(P; (t_i)_{i\in I}) \otimes_P F = K_{\bullet}(F; (t_i)_{i\in I})
    \end{equation*}
    is of the form 
    \begin{equation*}
        0 \to F \to \cdots \to \bigoplus_{1 \leq i_1 < \cdots < i_k \leq s} F \to \cdots \to \bigoplus_{i=1}^s F \to F \to 0, 
    \end{equation*}
    where the map $F \to F$ from index $(i_1, \ldots, i_k)$ to $(i_1, \ldots, \widehat{i_{\ell}}, \ldots, i_k)$ is the multiplication by $(-1)^{\ell + 1}t_{i_{\ell}}$. For $m \in \mathbb{Z}^r$, on the degree $-m$ part we have $(-1)^{\ell + 1}$ times the inclusion 
    \begin{equation*}
        F_{-m} \cap \mathcal{F}^{m_{i_1}+1}M \cap \cdots \cap \mathcal{F}^{m_{i_k}+1}M \to F_{-m} \cap \mathcal{F}^{m_{i_1}+1}M \cap \cdots \cap \widehat{\mathcal{F}^{m_{i_{\ell}}+1}M} \cap \cdots \cap \mathcal{F}^{m_{i_k}+1}M, 
    \end{equation*}
    with $F_{-m} = \mathcal{F}_1^{m_1}M \cap \cdots \cap \mathcal{F}_r^{m_r}M$. In other words, the degree $-m$ part of $K_{\bullet}(F; (t_i)_{i\in I})$ is 
    \begin{equation*}
        K_{\bullet}(F; (t_i)_{i\in I})_{-m} = C_{\bullet}(F_{-m}; (F_{-m} \cap \mathcal{F}^{m_i+1}M)_{i \in I}), 
    \end{equation*}
    and the right hand side is the complex in (2). 
\end{proof}

\begin{rem}
    The flatness of $\mathrm{Rees}_{\mathcal{F}_1, \ldots, \mathcal{F}_r}(M)$ over $\Bbbk[t_1, \ldots, t_r]$ is also equivalent to that $\mathcal{F}_1, \ldots, \mathcal{F}_r$ are \emph{compatible} filtrations on $M$. See \cite[\S15, Part 1]{MHM-project} for two other definitions for compatible filtrations, and the proofs that they are equivalent to the flatness of $\mathrm{Rees}_{\mathcal{F}_1, \ldots, \mathcal{F}_r}(M)$. 
\end{rem}

\begin{lem} \label{graded_of_compatible_filtration}
    Let $\Bbbk$ be a field, $R$ be a $\Bbbk$-algebra, and $M$ be an $R$-module. Suppose $\mathcal{F}_1, \ldots, \mathcal{F}_r$ are $\mathbb{Z}$-valued decreasing filtrations on $M$. Assume that $\mathcal{F}_i^0M = M$ for all $i = 1, \ldots, r$, and the extended Rees module $\mathrm{Rees}_{\mathcal{F}_1, \ldots, \mathcal{F}_r}(M)$ is flat over $\Bbbk[t_1, \ldots, t_r]$. Let $\alpha = (\alpha_1, \ldots, \alpha_r) \in \mathbb{R}_{> 0}^{r}$, and defined a decreasing $\mathbb{R}$-valued filtration $\mathcal{F}_{\alpha}$ on $M$ by 
    \begin{equation*}
        \mathcal{F}_{\alpha}^{\lambda}M = \sum_{\langle \alpha, m \rangle \geq \lambda} \mathcal{F}_1^{m_1}M \cap \cdots \cap \mathcal{F}_r^{m_r}M \subset M
    \end{equation*}
    for all $\lambda \in \mathbb{R}$, where the sum ranges over all $m = (m_1, \ldots, m_r) \in \mathbb{N}^r$ such that $\langle \alpha, m \rangle = \sum_{i=1}^{r} \alpha_i m_i \geq \lambda$. Then $\mathcal{F}_{\alpha}^0M = M$, and there is a canonical isomorphism 
    \begin{equation*}
        \mathcal{F}_{\alpha}^{\lambda}M/\mathcal{F}_{\alpha}^{>\lambda}M \simeq \bigoplus_{\langle \alpha, m \rangle = \lambda} \frac{\mathcal{F}_1^{m_1}M \cap \cdots \cap \mathcal{F}_r^{m_r}M}{(\mathcal{F}_1^{m_1+1}M \cap \cdots \cap \mathcal{F}_r^{m_r}M) + \cdots + (\mathcal{F}_1^{m_1}M \cap \cdots \cap \mathcal{F}_r^{m_r+1}M)}
    \end{equation*}
    for all $\lambda \in \mathbb{R}_{\geq 0}$, where $\mathcal{F}_{\alpha}^{>\lambda}M = \bigcup_{\lambda' > \lambda} \mathcal{F}_{\alpha}^{\lambda'}M$. Consequently, there is a canonical isomorphism 
    \begin{equation*}
        \mathrm{gr}_{\mathcal{F}_{\alpha}}(M) \simeq \mathrm{Rees}_{\mathcal{F}_1, \ldots, \mathcal{F}_r}(M)/(t_1,\ldots, t_r)\mathrm{Rees}_{\mathcal{F}_1, \ldots, \mathcal{F}_r}(M)
    \end{equation*}
    where $\mathrm{gr}_{\mathcal{F}_{\alpha}}(M) = \bigoplus_{\lambda \in \mathbb{R}_{\geq 0}} \mathcal{F}_{\alpha}^{\lambda}M/\mathcal{F}_{\alpha}^{>\lambda}M$. 
\end{lem}

\begin{proof}
    We consider the following ideals of $P = \Bbbk[t_1, \ldots, t_r]$. For $m \in \mathbb{N}^r$, let $I_m = (t^m) \subset P$ where we write $t^m = t_1^{m_1} \cdots t_r^{m_r}$. Fix $\alpha \in \mathbb{R}_{>0}^r$, for $\lambda \in \mathbb{R}$, let $J_{\lambda} = \sum_{\langle \alpha, m \rangle \geq \lambda} I_m$ and $J_{>\lambda} = \sum_{\langle \alpha, m \rangle > \lambda} I_m$. Then it is clear that 
    \begin{equation*}
        J_{\lambda}/J_{>\lambda} \simeq \bigoplus_{\langle \alpha, m \rangle = \lambda} \Bbbk \cdot t^m \simeq \bigoplus_{\langle \alpha, m \rangle = \lambda} I_m/(t_1, \ldots, t_r)I_m. 
    \end{equation*}
    Consider the extended Rees module $\mathcal{M} = \mathrm{Rees}_{\mathcal{F}_1, \ldots, \mathcal{F}_r}(M)$ graded by $\mathbb{Z}^r$. The degree $0$ part of $I_m\mathcal{M}$ is $\mathcal{F}_1^{m_1}M \cap \cdots \cap \mathcal{F}_r^{m_r}M$. Thus, $\mathcal{F}_{\alpha}^{\lambda}M$ and $\mathcal{F}_{\alpha}^{>\lambda}M$ are the degree $0$ part of $J_{\lambda}\mathcal{M}$ and $J_{>\lambda}\mathcal{M}$, respectively. It is clear that $\mathcal{F}_{\alpha}^0M = M$. Since $\mathcal{M}$ is flat over $P = \Bbbk[t_1, \ldots, t_r]$, we have 
    \begin{equation*}
        J_{\lambda}\mathcal{M}/J_{>\lambda}\mathcal{M} \simeq J_{\lambda}/J_{>\lambda} \otimes_P \mathcal{M} \simeq \bigoplus_{\langle \alpha, m \rangle = \lambda} I_m\mathcal{M}/(t_1, \ldots, t_r)I_m\mathcal{M}. 
    \end{equation*}
    Taking the degree $0$ part, we have the desired isomorphism
    \begin{equation*}
        \mathcal{F}_{\alpha}^{\lambda}M/\mathcal{F}_{\alpha}^{>\lambda}M \simeq \bigoplus_{\langle \alpha, m \rangle = \lambda} \frac{\mathcal{F}_1^{m_1}M \cap \cdots \cap \mathcal{F}_r^{m_r}M}{(\mathcal{F}_1^{m_1+1}M \cap \cdots \cap \mathcal{F}_r^{m_r}M) + \cdots + (\mathcal{F}_1^{m_1}M \cap \cdots \cap \mathcal{F}_r^{m_r+1}M)}. 
    \end{equation*}
    Note that $G = \mathcal{M}/(t_1, \ldots, t_r)\mathcal{M}$ is non-zero only in degrees $m \in \mathbb{N}^r$, and its degree $m$ part $G_m$ can be identified with the degree $0$ part of $I_m\mathcal{M}/(t_1, \ldots, t_r)I_m\mathcal{M}$. Hence 
    \begin{equation*}
        \mathrm{gr}_{\mathcal{F}_{\alpha}}(M) = \bigoplus_{\lambda \in \mathbb{R}_{\geq 0}} \mathcal{F}_{\alpha}^{\lambda}M/\mathcal{F}_{\alpha}^{>\lambda}M \simeq \bigoplus_{\lambda \in \mathbb{R}_{\geq 0}} \bigoplus_{\langle \alpha, m \rangle = \lambda} G_m \simeq \bigoplus_{m \in \mathbb{N}^r} G_m = \mathcal{M}/(t_1,\ldots, t_r)\mathcal{M}, 
    \end{equation*}
    where we rearrange the sum over $m$. 
\end{proof}

\subsection{Valuations} All valuations that we consider are real valuations. 

Let $X$ be an integral scheme with the function field $K(X)$. A real valuation on $K(X)$ is a map
\begin{equation*}
    v \colon K(X)^{\times} \to \mathbb{R}
\end{equation*}
such that $v(fg) = v(f) + v(g)$ and $v(f + g) \geq \min\{ v(f), v(g) \}$ for all $f, g \in K(X)^{\times}$, where we also set $v(0) = +\infty$. The valuation $v$ is centered at a point $x \in X$ if $v(f) \geq 0$ for all $f \in \mathscr{O}_{X,x}$, and $v(f) > 0$ for all $f \in \mathfrak{m}_{X,x}$, where $\mathfrak{m}_{X,x} \subset \mathscr{O}_{X,x}$ is the maximal ideal. 

Suppose $v$ is a valuation centered at $x \in X$. Suppose $D$ is a Cartier divisor on $X$ represented by some $f \in K(X)^{\times}$ on a neighborhood of $x$, then we define $v(D) = v(f)$. More generally, if $D$ is a $\mathbb{Q}$-Cartier $\mathbb{Q}$-divisor such that $mD$ is Cartier for some integer $m > 0$, then we define $v(D) = \frac{1}{m} v(mD)$. 

\begin{exmp}[Toric valuations] \label{simple_toric}
    Let $M$ be a lattice, that is, a free abelian group of finite rank, and $N$ be the dual lattice. Let $\sigma \subset N_{\mathbb{R}}$ be a strongly convex rational polyhedral cone, and $U = \mathop{\mathrm{Spec}} \Bbbk[\sigma^{\vee} \cap M]$ be the corresponding affine toric variety over a field $\Bbbk$. Write $\Bbbk[\sigma^{\vee} \cap M] = \bigoplus_{u \in \sigma^{\vee} \cap M} \Bbbk \cdot \chi^u$. Then every $\alpha \in \sigma$ gives a real valuation $v_{\alpha}$ on $U$ defined by 
    \begin{equation*}
        v_{\alpha}\left( \sum_{u} c_u\chi^u \right) = \min \{ \langle \alpha, u \rangle : c_u \neq 0 \}. 
    \end{equation*}
    Suppose $\sigma$ is a simplicial cone of full dimension in $N_{\mathbb{R}}$, that is, $\sigma$ is spanned by $e_1, \ldots, e_r \in N$ such that $e_1, \ldots, e_r$ is a basis of $N_{\mathbb{R}}$. Moreover, assume $e_i$ generates $(\mathbb{R}_{\geq 0} \cdot e_i) \cap N$. Then we have toric divisors $D_i$ on $U$ corresponding to $e_i$, and each $D_i$ is $\mathbb{Q}$-Cartier. Write $\alpha = \sum_{i=1}^{r} \alpha_i e_i \in \sigma$, then 
    \begin{equation*}
        v_{\alpha}(D_i) = \alpha_i
    \end{equation*}
    for all $i$. Thus, we identify the set of toric valuations on $U$ with $\mathbb{R}_{\geq 0}^r$. Note that if $(\alpha_1, \ldots, \alpha_r) \in \mathbb{R}_{>0}^r$, then $v_{\alpha}$ is centered at the unique torus fixed point $o \in U$, and $\{o\} = \bigcap_{i=1}^{r} D_i$. 
\end{exmp}

\begin{defn} \label{def_simple_toroidal_QM}
    Let $(X, D)$ be a pair, where $X$ is the spectrum of a local ring of dimension $r$ with the closed point $x \in X$, and $D = \sum_{i=1}^{r} D_i$ is a sum of distinct prime divisors with $\{x\} = \bigcap_{i=1}^{r} D_i$. The pair $(X, D)$ is said to be \emph{simple-toroidal} if there exists an snc pair $(X', D' = \sum_{i=1}^{r} D_i')$ with an action of a finite abelian group $G$ that preserves each $D_i$ and is free on $X' \smallsetminus D'$ such that $(X, D) \simeq (X', D')/G$. 

    In this case, the complete local ring $\hat{\mathscr{O}}_{X,x}$ is isomorphic to the complete local ring $\hat{\mathscr{O}}_{U,o}$ of an affine toric variety $U$ over $\kappa(x)$ corresponding to a simplicial cone $\sigma \subset N_{\mathbb{R}}$, where $o \in U$ is the torus fixed point, and the divisors $D_i$ correspond to the toric divisors on $U$. Choose a basis, we may assume $\sigma = \mathbb{R}_{\geq 0}^{r}$, and $\mathbb{Z}^r \subset N \subset \mathbb{R}^{r}$, so that $D_i$ corresponds to the standard basis $e_i$. Let $M \subset \mathbb{R}^r$ be the dual lattice of $N$, and $M_{\geq 0} = M \cap \mathbb{R}_{\geq 0}^{r}$. Then every $f \in \mathscr{O}_{X,x}$ can be written as a formal power series $f = \sum_{u \in M_{\geq 0}} c_u \chi^u$. For $\alpha = (\alpha_1, \ldots, \alpha_r) \in \mathbb{R}_{\geq 0}^{r}$, we have a valuation $v_{\alpha}$ on $X$ given by 
    \begin{equation*}
        v_{\alpha}(f) = \min \{ \langle \alpha, u \rangle : c_u \neq 0 \}. 
    \end{equation*}
    The set of all such valuations is denoted by $\mathrm{QM}_x(X, D)$, called \emph{quasi-monomial valuations} adapted to $(X, D)$. By definition, we have a bijection 
    \begin{equation*}
        \mathrm{QM}_x(X, D) \xrightarrow{\sim} \mathbb{R}_{\geq 0}^{r}, \quad v \mapsto (v(D_1), \ldots, v(D_r)). 
    \end{equation*}
    Let $\mathrm{QM}_x^{\circ}(X, D) \subset \mathrm{QM}_x(X, D)$ denote the subset of all $v_{\alpha}$ for which $\alpha \in \mathbb{R}_{>0}^r$. Every $v_{\alpha} \in \mathrm{QM}_x^{\circ}(X, D)$ is centered at the closed point $x \in X$. 
\end{defn}

\begin{defn}
    Let $X$ be a normal integral scheme, and $v$ be a real valuation on $X$ centered at $x \in X$. Let $\mathscr{L}$ be an invertible $\mathscr{O}_X$-module, and $s \in H^0(X, \mathscr{L})$. We define 
    \begin{equation*}
        v(s) = v(\phi(s|_U))
    \end{equation*}
    for any local trivialization $\phi \colon \mathscr{L}|_U \xrightarrow{\sim} \mathscr{O}_X|_U$ on an open neighborhood $U \subset X$ of $x$. Equivalently, if $D = \mathrm{div}_{\mathscr{L}}(s)$ is the effective Cartier divisor defined by $s$, then $v(s) = v(D)$. 

    More generally, suppose $L$ is a $\mathbb{Q}$-Cartier $\mathbb{Z}$-divisor, and $mL$ is Cartier for some integer $m > 0$. We define $v(s) = \frac{1}{m}v(s^m)$ for all $s \in H^0(X, \mathscr{O}_X(L))$, where $s^m \in H^0(X, \mathscr{O}_X(mL))$. 
    
    The \emph{filtration} $\mathcal{F}_v$ on $H^0(X, \mathscr{O}_X(L))$ induced by $v$ is the decreasing $\mathbb{R}$-valued filtration defined by 
    \begin{equation*}
        \mathcal{F}_v^{\lambda}H^0(X, \mathscr{O}_X(L)) = \{ s \in H^0(X, \mathscr{O}_X(L)) : v(s) \geq \lambda \}
    \end{equation*}
    for all $\lambda \in \mathbb{R}$. Note that each $\mathcal{F}_v^{\lambda}H^0(X, \mathscr{O}_X(L)) \subset H^0(X, \mathscr{O}_X(L))$ is a $H^0(X, \mathscr{O}_X)$-submodule. 

    We also have a filtration $\mathcal{F}_v$ on $\mathscr{O}_X(L)$ by coherent $\mathscr{O}_X$-submodules such that 
    \begin{equation*}
        H^0(U, \mathcal{F}_v^{\lambda}(\mathscr{O}_X(L))) = \left\{ \begin{array}{ll}
            \mathcal{F}_v^{\lambda}H^0(U, \mathscr{O}_X(L)) & \text{if}\ x \in U \\
            H^0(U, \mathscr{O}_X(L)) & \text{if}\ x \notin U
        \end{array} \right.
    \end{equation*}
    for any open subscheme $U \subset X$. Write $\mathfrak{a}_{\lambda}(v) = \mathcal{F}_v^{\lambda}\mathscr{O}_X$, called the valuative ideals associated with $v$. 
\end{defn}

\subsection{Simple-toroidal lc centers} Recall that an lc pair $(X, \Delta)$ is dlt if and only if it is snc at every lc center. A more general notion of \emph{qdlt} pairs is defined in \cite[Def.\ 35]{dFKX}: an lc pair $(X, \Delta)$ is qdlt if it is simple-toroidal at every lc center. 

We will work with a weaker condition, where $(X, \Delta)$ is an lc pair with $\lfloor \Delta \rfloor = \sum_{i=1}^{r} D_i$ such that all the lc centers of $(X, \Delta)$ are contained in $\mathrm{Supp}(\lfloor \Delta \rfloor)$ and each $D_i$ is a reduced divisor that is $\mathbb{Q}$-Cartier at the lc centers. Note that $X$ is of klt type. The following lemmas study some local properties for such pairs near an lc center. 

\begin{lem} \label{toroidal_lc_centers}
    Let $(X, \Delta)$ be an lc pair. Suppose $D_1, \ldots, D_r$ are $\mathbb{Q}$-Cartier reduced divisors on $X$ such that $D = D_1 + \cdots + D_r \leq \lfloor \Delta \rfloor$. Let $Z = \bigcap_{i=1}^{r} D_i$. Then the following hold: 
    \begin{enumerate}[label=\emph{(\arabic*)}, nosep]
        \item Every irreducible component of $Z$ has codimension $r$ in $X$. 
        \item Let $z \in Z$ be a generic point, $X_z = \mathop{\mathrm{Spec}}(\mathscr{O}_{X,z})$, and $D_{i,z} = D_i|_{X_z}$. Then 
        \begin{equation*}
            \Delta|_{X_z} = D_z = D_{1,z} + \cdots + D_{r,z}, 
        \end{equation*}
        and $(X_z, D_z)$ is simple-toroidal. 
    \end{enumerate}
\end{lem}

\begin{proof}
    Since each $D_i$ is $\mathbb{Q}$-Cartier, every irreducible component of $Z$ has codimension at most $r$. Note that $Z$ is an lc center of $(X, \Delta)$. Thus, by \cite[Prop.\ 34]{dFKX}, $Z$ has codimension $r$ and (2) holds. 
\end{proof}

\begin{lem} \label{Koszul_complex_div}
    Let $(X, \Delta)$ be an lc pair such that $X$ is of klt type. Let $D_1, \ldots, D_r$ be $\mathbb{Q}$-Cartier reduced divisors on $X$ such that $D_1 + \cdots + D_r \leq \lfloor \Delta \rfloor$, and $L$ be a $\mathbb{Q}$-Cartier $\mathbb{Z}$-divisor on $X$. Let 
    \begin{equation*}
        C_{\bullet} = C_{\bullet}\left( \mathscr{O}_X(L); (\mathscr{O}_X(L - D_i))_{i=1,\ldots, r} \right) 
    \end{equation*}
    be the complex as in Definition \ref{the_complex_C}. Then $H_p(C_{\bullet}) = 0$ for all $p > 0$. 
\end{lem}

\begin{proof}
    The problem is local, so we may assume that $X = \mathrm{Spec}(A)$ where $A$ is a local ring. Passing to a finite cover, we may assume that $L$ and all $D_i$ are Cartier. Write $D_i = \mathrm{div}(f_i)$ for some $f_i \in A$. Since $X$ is of klt type, $A$ is Cohen--Macaulay (see \cite[Cor.\ 2.88]{Kollar_singularity}). By Lemma \ref{toroidal_lc_centers}, $\bigcap_{i=1}^{r} D_i$ has codimension $r$ in $X$, hence $f_1, \ldots, f_r$ is a regular sequence of $A$ (see \cite[\texttt{02JN}]{stacks-project}). Then $C_{\bullet}$ is the Koszul complex $K_{\bullet}(A; f_1, \ldots, f_r)$. Hence $H_p(C_{\bullet}) = 0$ for all $p > 0$ (see \cite[\texttt{062F}]{stacks-project}). 
\end{proof}

\begin{lem} \label{res_of_divisor_to_center}
    Let $(X, \Delta)$ be an lc pair such that $X$ is of klt type. Let $D_1, \ldots, D_r$ be $\mathbb{Q}$-Cartier reduced divisors on $X$ such that $D_1 + \cdots + D_r \leq \lfloor \Delta \rfloor$, and $L$ be a $\mathbb{Q}$-Cartier $\mathbb{Z}$-divisor on $X$. Let $Z = \bigcap_{i=1}^{r} D_i$, and 
    \begin{equation*}
        \mathscr{L} \coloneqq \mathop{\mathrm{coker}} \left( \bigoplus_{i=1}^{r} \mathscr{O}_X(L - D_i) \to \mathscr{O}_X(L) \right) 
    \end{equation*}
    Then the following hold: 
    \begin{enumerate}[label=\emph{(\arabic*)},nosep]
        \item $\mathscr{L}$ is Cohen--Macaulay of pure codimension $r$ in $X$, and $\mathrm{Supp}(\mathscr{L}) \subset Z$. 
        \item Let $\zeta \in Z$ be a generic point. Then $\mathscr{L}_{\zeta} = \mathscr{O}_Y(L) \otimes \kappa(\zeta)$ if $L$ is Cartier at $\zeta$, and $\mathscr{L}_{\zeta} = 0$ if $L$ is not Cartier at $\zeta$. 
    \end{enumerate}
    In particular, $Z$ is Cohen--Macaulay and reduced. 
\end{lem}

\begin{proof}
    The problems are local, hence we may assume that $X = \mathop{\mathrm{Spec}}(A)$ where $(A, \mathfrak{m})$ is a local ring. 
    
    (1). Note that $\mathscr{I}_{D_i} = \mathscr{O}_X(-D_i) \subset \mathscr{O}_X$ is the ideal defining $D_i \subset X$ as a reduced closed subscheme. Thus, $Z = \bigcap_{i=1}^{r} D_i$ is defined by the ideal $\mathscr{I}_Z = \sum_{i=1}^{r} \mathscr{I}_{D_i} \subset \mathscr{O}_X$. Thus, if $L$ is a Cartier divisor, then the cokernel we get is $\mathscr{L} = \mathscr{O}_X(L)|_Z$. In general, $\mathscr{I}_{D_i} \cdot \mathscr{O}_X(L) \subset \mathscr{O}_X(L - D_i)$ since $\mathscr{O}_X(L - D_i)$ is the reflexive hull of $\mathscr{O}_X(-D_i) \otimes \mathscr{O}_X(L)$. Thus $\mathscr{I}_Z \cdot \mathscr{O}_X(L) \subset \sum_{i=1}^{r} \mathscr{O}_X(L-D_i)$. That is, $\mathscr{L}$ is annihilated by $\mathscr{I}_Z$, so $\mathrm{Supp}(\mathscr{L}) \subset Z$. 
    
    By Lemma \ref{toroidal_lc_centers}, $Z$ has pure codimension $r$ in $X$. Thus $\dim(\mathrm{Supp}(\mathscr{L})) \leq \dim(X) - r$. Let 
    \begin{equation*}
        C_{\bullet} = C_{\bullet}\left( \mathscr{O}_X(L); (\mathscr{O}_X(L - D_i))_{i=1,\ldots, r} \right) 
    \end{equation*}
    be the complex as in Definition \ref{the_complex_C}. Then $C_{\bullet}$ is quasi-isomorphic to $\mathscr{L}$ by Lemma \ref{Koszul_complex_div}. Then there is a spectral sequence for the local cohomology: 
    \begin{equation*}
        E_{1}^{p,q} = H^q_{\mathfrak{m}}(X, C_p) \Rightarrow H^q_{\mathfrak{m}}(X, \mathscr{L}). 
    \end{equation*}
    Each $C_p$ is a direct sum of modules of the form $\mathscr{O}_X(B)$ where $B$ is a $\mathbb{Q}$-Cartier $\mathbb{Z}$-divisor. Since $X$ of klt type, each $C_p$ is Cohen--Macaulay (see \cite[Cor.\ 2.88]{Kollar_singularity}), and $\mathrm{Supp}(C_p) = X$. Hence $H^q_{\mathfrak{m}}(X, C_p) = 0$ for all $q < \dim(X)$. Then the spectral sequence implies $H^q_{\mathfrak{m}}(X, \mathscr{L}) = 0$ for all $q < \dim(X) - r$. Thus 
    \begin{equation*}
        \mathop{\mathrm{depth}}(\mathscr{L}) \geq \dim(X) - r \geq \dim(\mathrm{Supp}(\mathscr{L})), 
    \end{equation*}
    see \cite[\texttt{0AVZ}]{stacks-project}. Since $\mathop{\mathrm{depth}}(\mathscr{L}) \leq \dim(\mathrm{Supp}(\mathscr{L}))$, all equalities hold. Hence $\mathscr{L}$ is Cohen--Macaulay, of pure codimension $r$ in $X$. 

    (2). We may further localize $X$ at $\zeta \in Z$. Then $(X, D)$ is simple-toroidal by Lemma \ref{toroidal_lc_centers}. Hence we are reduced to the toric case in Example \ref{simple_toric}, where $X = \mathop{\mathrm{Spec}} \kappa[\sigma^{\vee} \cap N]$ is the affine toric variety for a simplicial cone $\sigma \subset M_{\mathbb{R}}$. If $L$ is Cartier, then we may assume $L = 0$, and need to show $Z$ is the reduced fixed point. In this case, every $\chi^u$ for $u \in (\sigma^{\vee} \cap N) \smallsetminus \{0\}$ vanishes along some $D_i$, hence $\sum_{i=1}^{r} \mathscr{O}_X(-D_i)$ is the graded maximal ideal, and $\mathscr{O}_Z = \kappa$. In general, since the divisor class group of $X$ is generated by the toric divisors $D_i$, we may assume that $L = \sum_{i=1}^{r} c_iD_i$ for some $c_i \in \mathbb{Z}$. Suppose $u \in M$ and $\chi^u$ is section of $\mathscr{O}_X(L)$, that is, 
    \begin{equation*}
        \mathrm{div}(\chi^u) + L = \sum_{i=1}^{r} (\langle u, e_i \rangle + c_i) D_i \geq 0. 
    \end{equation*}
    If $L$ is not Cartier, then we must have $\langle u, e_i \rangle + c_i > 0$ for some $i$, hence $\chi^u$ is a section of $\mathscr{O}_X(L - D_i)$. Thus $\mathscr{O}_X(L) = \sum_{i=1}^{r} \mathscr{O}_X(L - D_i)$ in this case, and $\mathscr{L} = 0$.  

    Finally, for $L = 0$ we have $\mathscr{L} = \mathscr{O}_Z$. Then $Z$ is Cohen--Macaulay by (1), and generically reduced by (2). Hence $Z$ is reduced. 
\end{proof}

\begin{lem} \label{valuative_ideal_of_qm}
    Let $(X, \Delta)$ be an lc pair such that $X$ is of klt type. Let $D_1, \ldots, D_r$ be $\mathbb{Q}$-Cartier reduced divisors on $X$ such that $D = D_1 + \cdots + D_r \leq \lfloor \Delta \rfloor$. Assume $Z = \bigcap_{i=1}^{r} D_i$ is irreducible with the generic point $\zeta \in Z$. Suppose $\alpha = (\alpha_1, \ldots, \alpha_r) \in \mathbb{R}_{>0}^r$, and $v_{\alpha} \in \mathrm{QM}_{\zeta}^{\circ}(X, D)$ is the corresponding quasi-monomial valuation. Let $L$ be a $\mathbb{Q}$-Cartier $\mathbb{Z}$-divisor on $X$, and $\mathcal{F}_{v_{\alpha}}$ be the filtration on $\mathscr{O}_X(L)$ induced by $v_{\alpha}$. Then 
    \begin{equation*}
        \mathcal{F}_{v_{\alpha}}^{\lambda}(\mathscr{O}_X(L)) = \sum_{\langle \alpha, m \rangle \geq \lambda} \mathscr{O}_X(L - m_1D_1 - \cdots - m_rD_r), 
    \end{equation*}
    where the sum ranges over all $m = (m_1, \ldots, m_r) \in \mathbb{N}^r$ with $\langle \alpha, m \rangle = \sum_{i=1}^{r} \alpha_i m_i \geq \lambda$. 
\end{lem}

\begin{proof}
    First we show the conclusion holds after localization at $\zeta$. Thus we may assume that $(X, D)$ is simple-toroidal. It suffices to show the formula holds after completion, so we may assume that $(X, D)$ is toric as in Example \ref{simple_toric}. Then $X = \mathop{\mathrm{Spec}} \kappa[\sigma^{\vee} \cap M]$ for a simplicial cone $\sigma \subset N_{\mathbb{R}}$, and $D_1, \ldots, D_r$ are the toric divisors corresponding to $e_1, \ldots, e_r \in N$, which span $\sigma$. Since the divisor class group of $X$ is generated by the toric divisors $D_i$, we may assume that $L = \sum_{i=1}^{r} c_iD_i$ for some $c_i \in \mathbb{Z}$, and $nL$ is Cartier for some $n > 0$. Then $nL = \mathrm{div}(\chi^{u})$ for some $u \in M$. Suppose $s = \chi^t$ is a section for $\mathscr{O}_X(L)$, then 
    \begin{equation*}
        v_{\alpha}(s) = \frac{1}{n} v_{\alpha}(s^n) = \sum_{i=1}^{r} \alpha_i(\langle t, e_i \rangle + c_i)
    \end{equation*}
    Let $m_i = \langle t, e_i \rangle + c_i$, then $s \in \mathscr{O}_X(L - \sum_{i=1}^{r} m_iD_i)$. Hence the conclusion holds in the toric case. 

    In general, to check a section $s$ of $\mathscr{O}_X(L)$ is in $\mathcal{F}_{v_{\alpha}}^{\lambda}(\mathscr{O}_X(L))$, it suffices to check its germ at $\zeta$, thus we have 
    \begin{equation*}
        \sum_{\langle \alpha, m \rangle \geq \lambda} \mathscr{O}_X(L - m_1D_1 - \cdots - m_rD_r) \subset \mathcal{F}_{v_{\alpha}}^{\lambda}(\mathscr{O}_X(L)), 
    \end{equation*}
    and they coincide at $\zeta$. Note that we only have an inclusion since taking sums of ideals doesn't commute with contraction from a localization in general. In other words, there is a surjective map 
    \begin{equation*}
        \mathscr{O}_X(L) \bigg/ \sum_{\langle \alpha, m \rangle \geq \lambda} \mathscr{O}_X(L - m_1D_1 - \cdots - m_rD_r) \to \mathscr{O}_X \big/\mathcal{F}_{v_{\alpha}}^{\lambda}(\mathscr{O}_X(L))
    \end{equation*}
    such that the kernel has support strictly contained in $Z$. We claim that the generic point $\zeta$ of $Z$ is the only possible associated point of the left hand side, hence the kernel must be $0$. 

    We may assume that $X = \mathop{\mathrm{Spec}}(A)$ is affine. Let $\mathcal{F}_i$ be the filtration on $\mathscr{O}_X(L)$ given by $D_i$, that is, 
    \begin{equation*}
        \mathcal{F}_i^{k}(\mathscr{O}_X(L)) = \mathscr{O}_X(L - kD_i)
    \end{equation*}
    for all $k \geq 0$. Define the filtration $\mathcal{F}_{\alpha}$ on $\mathscr{O}_X(L)$ by 
    \begin{equation*}
        \mathcal{F}_{\alpha}^{\lambda}(\mathscr{O}_X(L)) = \sum_{\langle \alpha, m \rangle \geq \lambda} \mathcal{F}_1^{m_1}(\mathscr{O}_X(L)) \cap \cdots \cap \mathcal{F}_r^{m_r}(\mathscr{O}_X(L)) = \sum_{\langle \alpha, m \rangle \geq \lambda} \mathscr{O}_X(L - m_1D_1 - \cdots - m_rD_r). 
    \end{equation*}
    By Lemma \ref{Koszul_complex_div} and Lemma \ref{flatness_criterion_of_Rees}, the extended Rees module for $\mathcal{F}_1, \ldots, \mathcal{F}_r$ is flat over $\Bbbk[t_1, \ldots, t_r]$. Hence by Lemma \ref{graded_of_compatible_filtration}, each $\mathcal{F}_{\alpha}^{\lambda}(\mathscr{O}_X(L))/\mathcal{F}_{\alpha}^{>\lambda}(\mathscr{O}_X(L))$ is a direct sum of modules of the form 
    \begin{equation*}
        \mathscr{O}_X(L') \bigg/ \sum_{i=1}^r \mathscr{O}_X(L' - D_i), 
    \end{equation*}
    where $L' = L - m_1D_1 - \cdots - m_rD_r$. By Lemma \ref{res_of_divisor_to_center}, $\zeta$ is the only possible associated point for these modules. Hence the same holds for $\mathscr{O}_X(L)/\mathcal{F}_{\alpha}^{\lambda}(\mathscr{O}_X(L))$. 
\end{proof}

\section{Main theorems}

Let $\Bbbk$ be a field of characteristic $0$, and $S$ be an affine scheme essentially of finite type over $\Bbbk$. 

We will mainly work in the following situation: $(Y, \Delta)$ is an lc pair, $g \colon Y \to S$ is a proper morphism such that $-(K_Y + \Delta)$ is $g$-ample, and $E_1, \ldots, E_r$ are reduced divisors on $Y$ with $E = \sum_{i = 1}^{r} E_i = \lfloor \Delta \rfloor$, such that every lc center of $(Y, \Delta)$ is contained in $\mathrm{Supp}(E)$, and each $E_i$ is $\mathbb{Q}$-Cartier at the lc centers. The main object of this section is to prove Theorem \ref{main_thm_1} (Theorem \ref{finite_generation_thm}). 

\subsection{Auxiliary birational models}

We start with a lemma that allows us to perturb coefficients and make birational modifications of $(Y, \Delta)$ while preserving the assumptions. 

\begin{lem}[cf. {\cite[Lem.\ 2.3 and Prop.\ 3.6]{XZ_stable_deg}}] \label{adjust_coeff_of_Fano_type_model}
    Let $(Y, \Delta)$ be an lc pair, and $g \colon Y \to S$ be a proper morphism such that $-(K_Y + \Delta)$ is $g$-ample. Suppose 
    \begin{equation*}
        \begin{tikzcd}
            Y \ar[rr, "\phi", dashed] \ar[rd, "g"'] & & Y' \ar[ld, "g'"] \\
            & S &
        \end{tikzcd}
    \end{equation*}
    is a birational contraction over $S$ such that  $\phi$ is a local isomorphism at the lc centers of $(Y, \Delta)$, and $H'$ is a $g'$-ample divisor on $Y'$. Suppose $\Delta_0$ and $\Delta_1$ are effective $\mathbb{Q}$-divisor on $Y$ such that $\Delta_0 \leq \Delta$, $\Delta_0$ is $\mathbb{Q}$-Cartier at the lc centers of $(Y, \Delta)$, and $\mathrm{Supp}(\Delta_1)$ doesn't contain any lc centers of $(Y, \Delta)$. Then there exists $\epsilon, \delta \in \mathbb{Q}_{>0}$ and an effective $\mathbb{Q}$-divisor $\Delta'$ on $Y'$ such that $\Delta' \geq \phi_{*}(\Delta - \Delta_0 + \epsilon \Delta_1)$, $(Y', \Delta')$ is lc, the lc centers of $(Y', \Delta')$ are the images of lc centers of $(Y, \Delta)$ that are not contained in $\mathrm{Supp}(\Delta_0)$, and $-(K_{Y'} + \Delta') \sim_{S, \mathbb{Q}} \delta H'$. 
\end{lem}

\begin{proof}
     Choose a general divisor $A' \in |nH'|$ for an integer $n \gg 0$, and let $A = \phi^{-1}_{*}A'$. Then $A$ does not contain any lc centers of $(Y, \Delta)$ since $\phi$ is a local isomorphism there. 
     
     Let $H$ be a $g$-ample divisor on $Y$. Choose an integer $r > 0$ such that $r(\Delta_0 - \Delta_1 - A)$ is a $\mathbb{Z}$-divisor that is Cartier at the lc centers of $(Y, \Delta)$. Then, for $m \gg 0$, the base locus of $|mH + r(\Delta_0 - \Delta_1 - A)|$ does not contain any lc center of $(Y, \Delta)$ (see \cite[Lem.\ 2.4]{XZ_stable_deg}). Choose a divisor $G \in |mH + r(\Delta_0 - \Delta_1 - A)|$ that does not contain any lc centers of $(Y, \Delta)$, and consider 
     \begin{equation*}
         \Delta_{\epsilon} \coloneqq \Delta - \epsilon\Delta_0 + \epsilon(\Delta_1 + A + r^{-1}G). 
     \end{equation*}
     for $\epsilon \in \mathbb{Q}_{>0}$. Then 
     \begin{equation*}
         K_Y + \Delta_{\epsilon} \sim_{\mathbb{Q}} K_X + \Delta + \frac{\epsilon m}{r}H. 
     \end{equation*}
     In particular, $K_Y + \Delta_{\epsilon}$ is $\mathbb{Q}$-Cartier. Thus, for $\epsilon \ll 1$, we have $\Delta_{\epsilon} \geq \Delta - \Delta_0 + \epsilon (\Delta_1 + A)$, $(Y, \Delta_{\epsilon})$ is lc with lc centers those of $(Y, \Delta)$ that are not contained in $\mathrm{Supp}(\Delta_0)$, and $-(K_Y + \Delta_{\epsilon})$ is $g$-ample. Then we can choose $\Gamma_{\epsilon} \in |{-(K_Y + \Delta_{\epsilon})}|_{\mathbb{Q}}$ such that $(Y, \Delta_{\epsilon} + \Gamma_{\epsilon})$ is lc with the same lc centers as $(Y, \Delta_{\epsilon})$. 

     Then $(Y', \phi_{*}(\Delta_{\epsilon} + \Gamma_{\epsilon}))$ is lc since it is crepant birational to $(Y, \Delta_{\epsilon} + \Gamma_{\epsilon})$, and the lc centers of it are the images of the lc centers of $(Y, \Delta_{\epsilon})$. Let $\Delta' = \phi_{*}(\Delta_{\epsilon} + \Gamma_{\epsilon}) - \epsilon A'$, then $(Y', \Delta')$ is lc with the desired lc centers, and $-(K_{Y'} + \Delta') \sim_{S, \mathbb{Q}} \epsilon A' \sim_{S, \mathbb{Q}} (\epsilon/n)H'$. 
\end{proof}

The following is the key lemma. We show that for any $g$-pseudo-effective $\mathbb{Q}$-Cartier $\mathbb{Z}$-divisor $L$ on $Y$, we replace $Y$ with a birational model $Y'$ where we have the desired vanishing of higher cohomology. 

\begin{lem} \label{pass_to_sa_model}
    Let $(Y, \Delta)$ be an lc pair, $E_1, \ldots, E_r$ be reduced divisors on $Y$ with $E = \sum_{i = 1}^{r} E_i = \lfloor \Delta \rfloor$, and $g \colon Y \to S$ be a proper morphism such that $-(K_Y + \Delta)$ is $g$-ample. Assume every lc center of $(Y, \Delta)$ is contained in $\mathrm{Supp}(E)$, and each $E_i$ is $\mathbb{Q}$-Cartier at the lc centers. Suppose $L$ is a $g$-pseudo-effective $\mathbb{Q}$-Cartier $\mathbb{Z}$-divisor on $Y$. Then there exists a birational contraction over $S$ 
    \begin{equation*}
        \begin{tikzcd}
            Y \ar[rr, "\phi", dashed] \ar[rd, "g"'] & & Y' \ar[ld, "g'"] \\
            & S &
        \end{tikzcd}
    \end{equation*}
    such that the following hold: 
    \begin{enumerate}[label=\emph{(\arabic*)}, nosep]
        \item $L' \coloneqq \phi_{*}L$ is $g'$-semi-ample. 
        \item $E'_i = \phi_{*}E_i$ is $\mathbb{Q}$-Cartier for all $i = 1, \ldots, r$. 
        \item There exists an effective $\mathbb{Q}$-divisor $D'$ on $Y'$ such that $(Y', D')$ is lc, $E' = \sum_{i=1}^{r} E_i' = \lfloor D' \rfloor$, and $\phi$ is a local isomorphism over every lc center of $(Y', D')$. 
        \item Let $F = E_{i_1} + \cdots + E_{i_s}$ for a subset $\{i_1, \ldots, i_s\} \subset \{1, \ldots, r\}$, and $F' = \phi_{*}F$. Then 
        \begin{equation*}
            H^0(Y', \mathscr{O}_{Y'}(L' - F')) \simeq H^0(Y, \mathscr{O}_{Y}(L - F))
        \end{equation*}
         and $H^q(Y', \mathscr{O}_{Y'}(L' - F')) = 0$ for all $q > 0$. 
        \item The generic points of $Z' \coloneqq \bigcap_{i=1}^{r} E_i'$ are images of the generic points of $Z \coloneqq \bigcap_{i=1}^{r} E_i$ that are not contained in the stable base locus $\mathbf{B}(L)$, and $\phi$ is a local isomorphism at these points. 
    \end{enumerate}
\end{lem}

\begin{proof}
     By Lemma \ref{adjust_coeff_of_Fano_type_model} $Y$ is of klt Fano type over $S$. Hence we can run an $L$-MMP over $S$, which yields a semi-ample model for $L$ over $S$: 
     \begin{equation*}
        \begin{tikzcd}
            Y \ar[rr, "\psi", dashed] \ar[rd, "g"'] & & Y^{\mathrm{m}} \ar[ld, "g^{\mathrm{m}}"] \\
            & S &
        \end{tikzcd}
    \end{equation*}
    Note that $Y^{\mathrm{m}}$ is of klt type. Hence there exists a small birational modification $\pi \colon Y' \to Y^{\mathrm{m}}$ such that $E'_i \coloneqq \pi^{-1}_{*}E^{\mathrm{m}}_i$ is $\mathbb{Q}$-Cartier, and $\pi$ is a local isomorphism over the locus where every $E_i^{\mathrm{m}}$ is $\mathbb{Q}$-Cartier (see \cite[Lem.\ 4.7]{Zhuang_boundednessI}). We will show $\phi = \pi^{-1} \circ \psi \colon Y \dashrightarrow Y'$ satisfies the conditions. By the construction, (1) and (2) hold. 
    
    Choose a sufficiently small $\epsilon \in \mathbb{Q}_{>0}$ such that $\epsilon L - (K_Y + \Delta)$ is $g$-ample, and $G \in |\epsilon L - (K_Y + \Delta)|_{\mathbb{Q}}$ such that $(Y, \Delta + G)$ is lc with the same lc centers as $(Y, \Delta)$. Thus
    \begin{equation*}
        K_Y + \Delta + G \sim_{S, \mathbb{Q}} \epsilon L. 
    \end{equation*}
    Then $\psi$ is given by a $(K_Y + \Delta + G)$-MMP, so $(Y^{\mathrm{m}}, \Delta^{\mathrm{m}} + G^{\mathrm{m}})$ is lc, where $\Delta^{\mathrm{m}} = \psi_{*}\Delta$ and $G^{\mathrm{m}} = \psi_{*}G$, and $\psi$ is a local isomorphism over the lc centers of $(Y^{\mathrm{m}}, \Delta^{\mathrm{m}} + G^{\mathrm{m}})$ (see \cite[Cor.\ 1.23]{Kollar_singularity}). Then $\pi$ is also a local isomorphism over these lc centers since every $E^{\mathrm{m}}_i = \psi_{*}E_i$ is $\mathbb{Q}$-Cartier at these points. Let $\Delta' = \pi^{-1}_{*}\Delta^{\mathrm{m}}$ and $G' = \pi^{-1}_{*}G^{\mathrm{m}}$. Then $(Y', \Delta' + G')$ is lc, whose lc centers are the preimages of lc centers of $(Y^{\mathrm{m}}, \Delta^{\mathrm{m}} + G^{\mathrm{m}})$. Thus (3) holds with $D' = \Delta' + G'$. 

    To show (4), we may assume that $F = E$ by Lemma \ref{adjust_coeff_of_Fano_type_model}. Write $F = F_0 + F_1$ such that $F_0$ has no $\phi$-exceptional components, and $F_1$ is $\phi$-exceptional, so that $F' = \phi_{*}F_0$. Since $\phi$ is $L$-non-positive, 
    \begin{equation*}
        H^0(Y', \mathscr{O}_{Y'}(L' - F')) \simeq H^0(Y, \mathscr{O}_{Y}(L - F_0))
    \end{equation*}
    by Lemma \ref{compare_sections}. Every component of $F_1$ is contracted by $\phi$, hence is in the base locus of $L$ by the MMP. Thus, 
    \begin{equation*}
        H^0(Y', \mathscr{O}_{Y'}(L' - F')) \simeq H^0(Y, \mathscr{O}_{Y}(L - F_0)) = H^0(Y, \mathscr{O}_{Y}(L - F)). 
    \end{equation*}
    Since the lc centers of $(Y', \Delta' + G')$ are contained in $\mathrm{Supp}(E')$ and $E'$ is $\mathbb{Q}$-Cartier, $(Y', \Delta_0' + G')$ is klt, where $\Delta_0' = \Delta' - E' \geq 0$. We have 
    \begin{equation*}
        L' - E' \sim_{S, \mathbb{Q}} K_{Y'} + \Delta_0' + G' + (1-\epsilon)L'. 
    \end{equation*} 
    Note that $G' = \phi_{*}G$ is $g'$-big, hence we may write $G' \sim_{S, \mathbb{Q}} A' + B'$ such that $A'$ is $g'$-ample and $B' \geq 0$. Then 
    \begin{equation*}
        L' - E' \sim_{S, \mathbb{Q}} K_{Y'} + \Delta_0' + (1-t)G' + tB' + tA' + (1-\epsilon)L', 
    \end{equation*}
    where $(Y', \Delta_0' + (1-t)G' + tB')$ is klt for $0 < t \ll 1$, and $tA' + (1-\epsilon)L'$ is $g'$-ample. Thus 
    \begin{equation*}
        H^q(Y', \mathscr{O}_{Y'}(L' - F')) = 0
    \end{equation*}
    for all $q > 0$ by the Kawamata--Viehweg vanishing theorem. 

    Let $\zeta \in Z$ be a generic point. Then the $L$-MMP $\psi \colon Y \dashrightarrow Y^{\mathrm{m}}$ is a local isomorphism at $\zeta$ if and only if $\zeta \notin \mathbf{B}(L)$. Note that $\zeta$ is an lc center of $(Y, \Delta)$, hence each $E_i$ is $\mathbb{Q}$-Cartier at $\zeta$ by our assumptions. If $\psi$ is a local isomorphism at $\zeta$, then each $E_i^{\mathrm{m}}$ is $\mathbb{Q}$-Cartier at $\psi(\zeta)$, hence $\pi \colon Y' \to Y^{\mathrm{m}}$ is also a local isomorphism over $\psi(\zeta)$ by the construction. Suppose $\zeta'$ is a generic point of $Z'$, then it is an lc center of $(Y', D')$, so $\phi$ is a local isomorphism over $\zeta'$ by the condition (3). Then $\phi^{-1}(\zeta')$ is a generic point of $Z$. Thus, $\phi$ gives a bijection between the generic points of $Z$ that are not contained in $\mathbf{B}(L)$ and the generic points of $Z'$; moreover, $\phi$ is a local isomorphism at these points. 
\end{proof}

\subsection{The flatness of extended Rees modules} As mentioned in the Introduction, vanishing of higher cohomology allow us to deduce exactness of a Koszul complex, which implies the flatness of the extended Rees modules. We also keep track of the associated graded module. 

\begin{lem} \label{Koszul_exact_on_sections}
    Let $(Y, \Delta)$ be an lc pair, $E_1, \ldots, E_r$ be reduced divisors on $Y$ with $E = \sum_{i = 1}^{r} E_i = \lfloor \Delta \rfloor$, and $g \colon Y \to S$ be a proper morphism such that $-(K_Y + \Delta)$ is $g$-ample. Assume every lc center of $(Y, \Delta)$ is contained in $\mathrm{Supp}(E)$, and each $E_i$ is $\mathbb{Q}$-Cartier at the lc centers. Suppose $L$ is a $\mathbb{Q}$-Cartier $\mathbb{Z}$-divisor on $Y$. Let $Z = \bigcap_{i=1}^{r} E_i$, and 
    \begin{equation*}
        C_{\bullet} = C_{\bullet}\left( H^0(Y, \mathscr{O}_{Y}(L)); (H^0(Y, \mathscr{O}_Y(L - E_i)))_{i=1, \ldots, r} \right). 
    \end{equation*}
    denote the complex in Definition \ref{the_complex_C}. Then the following hold: 
    \begin{enumerate}[label=\emph{(\arabic*)}, nosep]
        \item $H_p(C_{\bullet}) = 0$ for all $p > 0$. 
        \item Let $Z^{(0, L)}$ denote the set of generic points $\zeta$ of $Z$ such that $\zeta \notin \mathbf{B}(L)$ and $L$ is Cartier at $\zeta$. The canonical map 
        \begin{equation*}
            H_0(C_{\bullet}) = H^0(Y, \mathscr{O}_{Y}(L)) \bigg/ \sum_{i=1}^r H^0(Y, \mathscr{O}_{Y}(L - E_i)) \to \prod_{\zeta \in Z^{(0, L)}} \mathscr{O}_Y(L) \otimes \kappa(\zeta), 
        \end{equation*}
        induced by the restrictions $H^0(Y, \mathscr{O}_Y(L)) \to \mathscr{O}_Y(L) \otimes \kappa(\zeta)$, is injective. 
    \end{enumerate}
\end{lem}

\begin{proof}
    If $L$ is not $g$-pseudo-effective, then $H^0(Y, \mathscr{O}_Y(L)) = 0$, so the conclusion is trivial. Assume that $L$ is $g$-pseudo-effective. Let 
    \begin{equation*}
        \begin{tikzcd}
            Y \ar[rr, "\phi", dashed] \ar[rd, "g"'] & & Y' \ar[ld, "g'"] \\
            & S &
        \end{tikzcd}
    \end{equation*}
    be the birational contraction given by Lemma \ref{pass_to_sa_model}, with $L' = \phi_{*}L$ and $E'_i = \phi_{*}E_i$. Note that each $C_p$ is a direct sum of terms of the form 
    \begin{equation*}
        H^0(Y, \mathscr{O}_Y(L - E_{i_1})) \cap \cdots \cap H^0(Y, \mathscr{O}_Y(L - E_{i_p})) = H^0(Y, \mathscr{O}_Y(L - E_{i_1} - \cdots - E_{i_p})). 
    \end{equation*}
    Thus, since $H^0(Y, \mathscr{O}_Y(L - E_{i_1} - \cdots - E_{i_p})) \simeq H^0(Y', \mathscr{O}_{Y'}(L' - E'_{i_1} - \cdots - E'_{i_p}))$ by Lemma \ref{pass_to_sa_model}, we have 
    \begin{equation*}
        C_{\bullet} = C_{\bullet}\left( H^0(Y', \mathscr{O}_{Y'}(L')); (H^0(Y', \mathscr{O}_{Y'}(L' - E'_i)))_{i} \right). 
    \end{equation*}
    Let $\mathscr{C}_{\bullet} = C_{\bullet}(\mathscr{O}_{Y'}(L'); (\mathscr{O}_{Y'}(L' - E'_i))_{i})$. There is a second-quadrant spectral sequence 
    \begin{equation*}
        E_1^{p,q} = H^{q}(Y', \mathscr{C}_{-p}) \Rightarrow \mathbb{H}^{p+q}(Y', \mathscr{C}_{\bullet})
    \end{equation*}
    Since $H^q(Y', \mathscr{C}_{-p}) = 0$ for all $p$ and $q > 0$ by Lemma \ref{pass_to_sa_model}, and $K_p = H^0(Y', \mathscr{C}_p)$, we have 
    \begin{equation*}
        H_p(C_{\bullet}) \simeq \mathbb{H}^{-p}(Y', \mathscr{C}_{\bullet}). 
    \end{equation*}
    By Lemma \ref{Koszul_complex_div}, $H_p(\mathscr{C}_{\bullet}) = 0$ for $p > 0$, so $\mathscr{C}_{\bullet}$ is quasi-isomorphic to the sheaf 
    \begin{equation*}
        \mathscr{L}' = \mathscr{O}_{Y'}(L') \bigg/ \sum_{i=1}^r \mathscr{O}_{Y'}(L' - E'_i). 
    \end{equation*}
    Thus $\mathbb{H}^{-p}(Y', \mathscr{C}_{\bullet}) \simeq H^{-p}(Y', \mathscr{L}') = 0$ for all $p > 0$, since a sheaf has no cohomology in negative degrees. This shows (1). 

    Moreover, $H_0(C_{\bullet}) = H^0(Y', \mathscr{L}')$. Let $Z' = \bigcap_{i=1}^{r} E'_i$, and $Z'^{(0, L')}$ denote the set of generic points of $Z'$ where $L'$ is Cartier. By Lemma \ref{res_of_divisor_to_center}, $\mathscr{L}'$ is a Cohen--Macaulay module such that the generic points of $\mathrm{Supp}(\mathscr{L}')$ are exactly $Z'^{(0, L')}$, and $\mathscr{L}'_{\zeta'} = \mathscr{O}_{Y'}(L') \otimes \kappa(\zeta')$ for every $\zeta' \in Z'^{(0, L')}$. Then the canonical restriction map 
    \begin{equation*}
        H^0(Y', \mathscr{L}') \to \prod_{\zeta' \in Z'^{(0, L')}} \mathscr{L}'_{\zeta'} = \prod_{\zeta' \in Z'^{(0, L')}} \mathscr{O}_{Y'}(L') \otimes \kappa(\zeta')
    \end{equation*}
    is injective. By Lemma \ref{pass_to_sa_model}, $Z'^{(0,L')}$ is the image of $Z^{(0,L)}$, and $\phi$ is a local isomorphism at these points. Thus the map above is identified with the map 
    \begin{equation*}
        H_0(C_{\bullet}) = H^0(Y, \mathscr{O}_{Y}(L)) \bigg/ \sum_{i=1}^r H^0(Y, \mathscr{O}_{Y}(L - E_i)) \to \prod_{\zeta \in Z^{(0, L)}} \mathscr{O}_Y(L) \otimes \kappa(\zeta)
    \end{equation*}
    induced by restrictions. This shows (2). 
\end{proof}

\begin{thm} \label{Rees_module_flat}
    Let $(Y, \Delta)$ be an lc pair, $E_1, \ldots, E_r$ be reduced divisors on $Y$ with $E = \sum_{i = 1}^{r} E_i = \lfloor \Delta \rfloor$, and $g \colon Y \to S$ be a proper morphism such that $-(K_Y + \Delta)$ is $g$-ample. Assume every lc center of $(Y, \Delta)$ is contained in $\mathrm{Supp}(E)$, and each $E_i$ is $\mathbb{Q}$-Cartier at the lc centers. For each $m = (m_1, \ldots, m_r) \in \mathbb{Z}^r$, write 
    \begin{equation*}
        E(m) \coloneqq \sum_{i=1}^{r} \max(m_i, 0) E_i. 
    \end{equation*}
    Suppose $L$ is a $\mathbb{Q}$-Cartier $\mathbb{Z}$-divisor on $Y$. Then the module 
    \begin{equation*}
        M = \bigoplus_{m \in \mathbb{Z}^r} H^0(Y, \mathscr{O}_Y(L - E(m))) t_1^{-m_1} \cdots t_r^{-m_r} \subset H^0(Y, \mathscr{O}_Y(L))[t_1^{\pm 1}, \ldots, t_r^{\pm 1}]
    \end{equation*}
    is flat over $\Bbbk[t_1, \ldots, t_r]$. 
\end{thm}

\begin{proof}
    Apply Lemma \ref{adjust_coeff_of_Fano_type_model} for $\phi = \mathrm{id}$, $\Delta_0 = E$, and $\Delta_1 = 0$, we conclude that $Y$ is of klt type. Then there is a small birational modification $\pi \colon Y' \to Y$ such that $E_i' = \pi^{-1}_{*}E_i$ is $\mathbb{Q}$-Cartier for each $i$, and $\pi$ is a local isomorphism over the locus where every $E_i$ is $\mathbb{Q}$-Cartier (see \cite[Lem.\ 4.7]{Zhuang_boundednessI}). Apply Lemma \ref{adjust_coeff_of_Fano_type_model} for $\phi = \pi^{-1}$ and $\Delta_0 = \Delta_1 = 0$, we get a pair $(Y', \Delta')$ satisfying the same conditions as in the Theorem. By Lemma \ref{compare_sections}, we may replace $(Y, \Delta)$ with $(Y', \Delta')$, hence assume every $E_i$ is $\mathbb{Q}$-Cartier. 
    
    For each $i$, let $\mathcal{F}_{E_i}$ be the filtration on $H^0(Y, \mathscr{O}_Y(L))$ given by 
    \begin{equation*}
        \mathcal{F}_{E_i}^nH^0(Y, \mathscr{O}_{Y}(L)) = \left\{ \begin{array}{ll}
            H^0(Y, \mathscr{O}_Y(L)) & \text{if}\ n \leq 0,  \\
            H^0(Y, \mathscr{O}_Y(L - nE_i)) & \text{if}\ n > 0. 
        \end{array} \right. 
    \end{equation*}
    Then $M$ is the extended Rees module for the filtrations $\mathcal{F}_{E_1}, \ldots, \mathcal{F}_{E_r}$ on $H^0(Y, \mathscr{O}_Y(L))$. By Lemma \ref{flatness_criterion_of_Rees}, it suffices to show that for every $I = \{i_1, \ldots, i_s\} \subset \{1, \ldots, r\}$ and $m \in \mathbb{Z}^r$, for the complex
    \begin{equation*}
        C_{\bullet} = C_{\bullet}\left( H^0(Y, \mathscr{O}_{Y}(L - E(m))); (H^0(Y, \mathscr{O}_Y(L - E(m) - E_i)))_{i \in I} \right)
    \end{equation*}
    as in Definition \ref{the_complex_C}, we have $H_p(C_{\bullet}) = 0$ for all $p > 0$. 

    By Lemma \ref{adjust_coeff_of_Fano_type_model}, there exists an effective $\mathbb{Q}$-divisor $\Delta_I$ on $Y$ such that $\lfloor \Delta_I \rfloor = \sum_{i \in I} E_i$, $(Y, \Delta_I)$ is lc with lc centers those of $(Y, \Delta)$ that are not contained in $\sum_{j \notin I} E_j$, and $-(K_Y + \Delta_I)$ is $g$-ample. Hence we can apply Lemma \ref{Koszul_exact_on_sections} to $E_i$ for $i \in I$ and $L - E(m)$, and conclude that $H_p(C_{\bullet}) = 0$ for all $p > 0$. 
\end{proof}

\subsection{The graded algebra} Now we deduce Theorem \ref{the_central_fiber_of_Rees_alg} from results in the previous subsection. Note that we can get ring properties of the associated graded algebra since the canonical map in Lemma \ref{Koszul_exact_on_sections}(2) is compatible with multiplications. 

\begin{thm} \label{the_central_fiber_of_Rees_alg}
    Let $(Y, \Delta)$ be an lc pair, $E_1, \ldots, E_r$ be reduced divisors on $Y$ with $E = \sum_{i = 1}^{r} E_i = \lfloor \Delta \rfloor$, and $g \colon Y \to S$ be a proper morphism such that $-(K_Y + \Delta)$ is $g$-ample. Assume every lc center of $(Y, \Delta)$ is contained in $\mathrm{Supp}(E)$, and each $E_i$ is $\mathbb{Q}$-Cartier at the lc centers. For each $m = (m_1, \ldots, m_r) \in \mathbb{Z}^r$, write 
    \begin{equation*}
        E(m) \coloneqq \sum_{i=1}^{r} \max(m_i, 0) E_i. 
    \end{equation*}
    Let $L_1, \ldots, L_s$ be $\mathbb{Q}$-Cartier $\mathbb{Z}$-divisors on $Y$, 
    \begin{equation*}
        R = R(L_1, \ldots, L_s) = \bigoplus_{n = (n_1, \ldots, n_s) \in \mathbb{N}^s}H^0(Y, \mathscr{O}_Y(n_1L_1 + \cdots + n_sL_s))
    \end{equation*}
    be the multi-section ring, and 
    \begin{equation*}
        \mathcal{R} = \bigoplus_{n \in \mathbb{N}^s} \bigoplus_{m \in \mathbb{Z}^r} H^0(Y, \mathscr{O}_Y(n_1L_1 + \cdots + n_sL_s - E(m))) t_1^{-m_1} \cdots t_r^{-m_r} \subset R[t_1^{\pm 1}, \ldots, t_r^{\pm r}]
    \end{equation*}
    be the extended Rees algebra. Then the following hold: 
    \begin{enumerate}[label=\emph{(\arabic*)}, nosep]
        \item $\mathcal{R}$ is of finite type over $A = H^0(S, \mathscr{O}_S)$. 
        \item $\mathcal{R}$ is flat over $\Bbbk[t_1, \ldots, t_r]$. 
        \item Let $Z = \bigcap_{i = 1}^{r} E_i$, with the total ring of fractions $K(Z)$. Then there is an injective ring map 
        \begin{equation*}
            \mathcal{R}/(t_1, \ldots, t_r)\mathcal{R} \hookrightarrow K(Z)[\mathbb{N}^s \times \mathbb{N}^r]. 
        \end{equation*}
        In particular, if $Z$ is irreducible, then $\mathcal{R}/(t_1, \ldots, t_r)\mathcal{R}$ is an integral domain. 
    \end{enumerate}
\end{thm}

\begin{proof}
    Apply Lemma \ref{adjust_coeff_of_Fano_type_model} for $\phi = \mathrm{id}$, $\Delta_0 = E$, and $\Delta_1 = 0$, we conclude that $Y$ is of klt type. Then there is a small birational modification $\pi \colon Y' \to Y$ such that $E_i' = \pi^{-1}_{*}E_i$ is $\mathbb{Q}$-Cartier for each $i$, and $\pi$ is a local isomorphism over the locus where every $E_i$ is $\mathbb{Q}$-Cartier (see \cite[Lem.\ 4.7]{Zhuang_boundednessI}). Apply Lemma \ref{adjust_coeff_of_Fano_type_model} for $\phi = \pi^{-1}$ and $\Delta_0 = \Delta_1 = 0$, we get a pair $(Y', \Delta')$ satisfying the same conditions as in the Theorem. By Lemma \ref{compare_sections}, we may replace $(Y, \Delta)$ with $(Y', \Delta')$, hence assume every $E_i$ is $\mathbb{Q}$-Cartier. 
    
    (1). One can write $\mathcal{R}$ as a quotient of a Cox ring associated with some $\mathbb{Q}$-Cartier divisors on $Y$. Thus $\mathcal{R}$ is of finite type over $A = H^0(S, \mathscr{O}_S)$ since $Y$ is of klt Fano type over $S$. 

    (2). This follows from Theorem \ref{Rees_module_flat}. 

    (3). Let $Z^{(0)}$ be the set of generic points of $Z$, so that $K(Z) = \prod_{\zeta \in Z^{(0)}} \kappa(\zeta)$. Note that 
    \begin{equation*}
        \mathcal{R}/(t_1, \ldots, t_r)\mathcal{R} = \bigoplus_{n \in \mathbb{N}^s} \bigoplus_{m \in \mathbb{N}^r} H^0(Y, \mathscr{O}_{Y}(L(n,m))) \bigg/ \sum_{i=1}^r H^0(Y, \mathscr{O}_{Y}(L(n,m) - E_i)), 
    \end{equation*}
    where $L(n,m) = \sum_{j=1}^{s} n_jL_j - \sum_{i=1}^{r} m_iE_i$ for $(n,m) \in \mathbb{N}^s \times \mathbb{N}^r$. Let $Z^{(0, L(n,m))} \subset Z^{(0)}$ denote the set of $\zeta \in Z^{(0)}$ such that $\zeta \notin \mathbf{B}(L(n,m))$ and $L(n,m)$ is Cartier at $\zeta$. Then the canonical map 
    \begin{equation*}
        H^0(Y, \mathscr{O}_{Y}(L(n,m))) \bigg/ \sum_{i=1}^r H^0(Y, \mathscr{O}_{Y}(L(n,m) - E_i)) \to \prod_{\zeta \in Z^{(0, L(n,m))}} \mathscr{O}_Y(L(n,m)) \otimes \kappa(\zeta)
    \end{equation*}
    is injective. For each $\zeta \in Z^{(0)}$, let $\Sigma_{\zeta} \subset \mathbb{N}^s \times \mathbb{N}^r$ denote the submonoid consisting of $(n,m)$ such that $L(n,m)$ is Cartier at $\zeta$. Then, the canonical map 
    \begin{equation*}
        \mathcal{R}/(t_1, \ldots, t_r)\mathcal{R} \to \prod_{\zeta \in Z^{(0)}} \bigoplus_{(n,m) \in \Sigma_{\zeta}} \mathscr{O}_Y(L(n,m)) \otimes \kappa(\zeta)
    \end{equation*}
    is injective. It is a ring homomorphism since the multiplications on both sides are given by multiplication of sections, and the map is induced by restricting sections. Each $\mathscr{O}_Y(L(n,m)) \otimes \kappa(\zeta)$ is an $1$-dimensional $\kappa(\zeta)$-vector space, and the multiplications 
    \begin{equation*}
        (\mathscr{O}_Y(L(n,m)) \otimes \kappa(\zeta)) \otimes_{\kappa(\zeta)} (\mathscr{O}_Y(L(n',m')) \otimes \kappa(\zeta)) \to \mathscr{O}_Y(L(n+n',m+m')) \otimes \kappa(\zeta)
    \end{equation*}
    are isomorphisms for all $(n,m), (n',m') \in \Sigma_{\zeta}$. Hence we have an isomorphism 
    \begin{equation*}
        \bigoplus_{(n,m) \in \Sigma_{\zeta}} \mathscr{O}_Y(L(n,m)) \otimes \kappa(\zeta) \simeq \kappa(\zeta)[\Sigma_{\zeta}], 
    \end{equation*}
    where $\kappa(\zeta)[\Sigma_{\zeta}]$ denotes the monoid ring on $\Sigma_{\zeta}$. Then the inclusion $\Sigma_{\zeta} \hookrightarrow \mathbb{N}^s \times \mathbb{N}^r$ induces an injective map $\kappa(\zeta)[\Sigma_{\zeta}] \to \kappa(\zeta)[\mathbb{N}^s \times \mathbb{N}^r]$. Take the composition of all these maps, we get an injective ring map 
    \begin{equation*}
        \mathcal{R}/(t_1, \ldots, t_r)\mathcal{R} \hookrightarrow \prod_{\zeta \in Z^{(0)}} \kappa(\zeta)[\mathbb{N}^s \times \mathbb{N}^r] = K(Z)[\mathbb{N}^s \times \mathbb{N}^r]. 
    \end{equation*}
    If $Z$ is irreducible, then $K(Z)$ is the function field of $Z$, and $K(Z)[\mathbb{N}^s \times \mathbb{N}^r]$ is a polynomial ring on $s + r$ variables. In particular, $K(Z)[\mathbb{N}^s \times \mathbb{N}^r]$ is an integral domain, and so is $\mathcal{R}/(t_1, \ldots, t_r)\mathcal{R}$. 
\end{proof}

The finite generation for valuations follows from the observation that if the extended Rees algebra $\mathcal{R}$ is flat $\Bbbk[t_1, \ldots, t_r]$, and $\mathcal{R}/(t_1, \ldots, t_r)$ is an integral domain, then it coincide with the graded algebra of the valuation. We make this observation precise in the following lemma (cf. \cite[Lem.\ 4.8]{XZ_stable_deg} and \cite[Prop.\ 5.22]{Xu_K-Stability_Book}). To have a more concise statement and emphasize the essential assumptions, we consider a single module $H^0(Y, \mathscr{O}_Y(L))$; see also Remark \ref{get_rid_of_ample} for the use of an auxiliary ample divisor. 

\begin{lem} \label{integral_central_fiber_is_valuative}
    Let $(Y, \Delta)$ be an lc pair such that $Y$ is of klt type, $E_1, \ldots, E_r$ be $\mathbb{Q}$-Cartier reduced divisors on $Y$ with $E = \sum_{i = 1}^{r} E_i = \lfloor \Delta \rfloor$, and $g \colon Y \to S$ be a proper morphism. Assume $Z = \bigcap_{i=1}^{r} E_i$ is irreducible with the generic point $\zeta \in Z$. For each $m = (m_1, \ldots, m_r) \in \mathbb{Z}^r$, write 
    \begin{equation*}
        E(m) \coloneqq \sum_{i=1}^{r} \max(m_i, 0) E_i. 
    \end{equation*}
    Suppose $L$ and $H$ are $\mathbb{Q}$-Cartier $\mathbb{Z}$-divisors on $Y$, and $H$ is $g$-ample. Let 
    \begin{equation*}
        R = \bigoplus_{k \in \mathbb{N}} H^0(Y, \mathscr{O}_Y(kH)), \quad M = \bigoplus_{k \in \mathbb{N}} H^0(Y, \mathscr{O}_Y(L + kH)). 
    \end{equation*}
    Let 
    \begin{equation*}
        \mathcal{R} = \bigoplus_{k \in \mathbb{N}} \bigoplus_{m \in \mathbb{Z}^r} H^0(Y, \mathscr{O}_Y(kH - E(m))) t_1^{-m_1} \cdots t_r^{-m_r} \subset R[t_1^{\pm 1}, \ldots, t_r^{\pm r}], 
    \end{equation*}
    and 
    \begin{equation*}
        \mathcal{M} = \bigoplus_{k \in \mathbb{N}} \bigoplus_{m \in \mathbb{Z}^r} H^0(Y, \mathscr{O}_Y(L + kH - E(m))) t_1^{-m_1} \cdots t_r^{-m_r} \subset M[t_1^{\pm 1}, \ldots, t_r^{\pm r}]. 
    \end{equation*}
    Assume that $\mathcal{R}$ and $\mathcal{M}$ are flat over $\Bbbk[t_1, \ldots, t_r]$, $\overline{\mathcal{R}} = \mathcal{R}/(t_1,\ldots,t_r)\mathcal{R}$ is an integral domain, and $\overline{\mathcal{M}} = \mathcal{M}/(t_1,\ldots,t_r)\mathcal{M}$ is a torsion-free $\overline{\mathcal{R}}$-module. Suppose $\alpha = (\alpha_1, \ldots, \alpha_r)\in \mathbb{R}_{>0}^r$ and $v_{\alpha} \in \mathrm{QM}^{\circ}_{\zeta}(Y, E)$ is the corresponding quasi-monomial valuation. Let $\mathcal{F}_{v_{\alpha}}$ denote the filtration on $H^0(Y, \mathscr{O}_Y(L))$ induced by $v_{\alpha}$. Then 
    \begin{equation*}
        \mathcal{F}_{v_{\alpha}}^{\lambda} H^0(Y, \mathscr{O}_Y(L)) = \sum_{\langle \alpha, m \rangle \geq \lambda} H^0(Y, \mathscr{O}_Y(L - E(m)))
    \end{equation*}
    for all $\lambda \geq 0$, where the sum ranges over all $m \in \mathbb{N}^r$ such that $\langle \alpha, m \rangle = \sum_{i=1}^{r} \alpha_i m_i \geq \lambda$. 
\end{lem}

\begin{proof}
    For any $\mathbb{Q}$-Cartier $\mathbb{Z}$-divisor $D$ on $Y$, let $\mathcal{F}_{\alpha}$ denote the filtration on $H^0(Y, \mathscr{O}_Y(D))$ given by 
    \begin{equation*}
        \mathcal{F}_{\alpha}^{\lambda}H^0(Y, \mathscr{O}_Y(D)) = \sum_{\langle \alpha, m \rangle \geq \lambda} H^0(Y, \mathscr{O}_Y(D - E(m)))
    \end{equation*}
    for all $\lambda \geq 0$. If $s \in H^0(Y, \mathscr{O}_Y(D))$, defined 
    \begin{equation*}
        \mathrm{ord}_{\alpha}(s) = \sup \{ \lambda \in \mathbb{R}_{\geq 0} : s \in \mathcal{F}_{\alpha}^{\lambda}H^0(Y, \mathscr{O}_Y(D)) \}. 
    \end{equation*}
    Note that the supremum is also a maximum, that is, if $\mathrm{ord}_{\alpha}(s) = \lambda$, then $s \in \mathcal{F}_{\alpha}^{\lambda}H^0(Y, \mathscr{O}_Y(D))$. Now we have filtrations $\mathcal{F}_{\alpha}$ on $R_k = H^0(Y, \mathscr{O}_Y(kH))$ and $M_k = H^0(Y, \mathscr{O}_Y(L + kH))$, and filtrations $\mathcal{F}_{v_{\alpha}}$ induced by the valuation $v_{\alpha}$. Then we need to show $\mathcal{F}_{\alpha}^{\lambda}M_0 = \mathcal{F}_{v_{\alpha}}^{\lambda}M_0$, or equivalently, 
    \begin{equation*}
        v_{\alpha}(s) = \mathrm{ord}_{\alpha}(s)
    \end{equation*}
    for all $s \in M_0$. 

    Note that $v_{\alpha}$ is multiplicative: if $f \in R_k$ and $s \in M_{\ell}$, then $v_{\alpha}(fs) = v_{\alpha}(f) + v_{\alpha}(s)$ where $fs \in M_{k+\ell}$. We will show that $\mathrm{ord}_{\alpha}$ is also multiplicative. Let $f \in R_k$ and $s \in M_{\ell}$ with $\mathrm{ord}_{\alpha}(f) = \lambda$ and $\mathrm{ord}_{\alpha}(s) = \mu$, since 
    \begin{equation*}
        \mathcal{F}_{\alpha}^{\lambda} R_k \cdot \mathcal{F}_{\alpha}^{\mu} M_{\ell} \subset \mathcal{F}_{\alpha}^{\lambda + \mu} M_{k + \ell}, 
    \end{equation*}
    we have $\mathrm{ord}_{\alpha}(fs) \geq \lambda + \mu$. By Lemma \ref{graded_of_compatible_filtration}, we have 
    \begin{equation*}
        \mathrm{gr}_{\mathcal{F}_{\alpha}}(R) \simeq \overline{\mathcal{R}}, \quad \mathrm{gr}_{\mathcal{F}_{\alpha}}(M) \simeq \overline{\mathcal{M}}, 
    \end{equation*}
    hence $\mathrm{gr}_{\mathcal{F}_{\alpha}}(R)$ is an integral domain and $\mathrm{gr}_{\mathcal{F}_{\alpha}}(M)$ is a torsion-free $\mathrm{gr}_{\mathcal{F}_{\alpha}}(R)$-module by our assumptions. Let $\bar{f} \in \mathcal{F}_{\alpha}^{\lambda}R_k/\mathcal{F}_{\alpha}^{>\lambda}R_k$ and $\bar{s} \in \mathcal{F}_{\alpha}^{\mu}M_{\ell}/\mathcal{F}_{\alpha}^{>\mu}M_{\ell}$ be the images of $f$ and $s$, respectively. Then $\bar{f}$ and $\bar{s}$ are nonzero, hence $\bar{f} \cdot \bar{s} \in \mathrm{gr}_{\mathcal{F}_{\alpha}}(M)$ is also nonzero. By definition, the product $\bar{f} \cdot \bar{s} \in \mathrm{gr}_{\mathcal{F}_{\alpha}}(M)$ is the image of $fs \in \mathcal{F}_{\alpha}^{\mu} M_{k + \ell}$ in $\mathcal{F}_{\alpha}^{\lambda + \mu} M_{k + \ell}/\mathcal{F}_{\alpha}^{>\lambda + \mu} M_{k + \ell}$. Hence $fs \notin \mathcal{F}_{\alpha}^{>\lambda + \mu} M_{k + \ell}$, that is, $\mathrm{ord}_{\alpha}(fs) = \lambda + \mu$. 

    The valuation $v_{\alpha}$ induces filtrations on $\mathscr{O}_Y(D)$ for any $\mathbb{Q}$-Cartier $\mathbb{Z}$-divisor $D$, such that 
    \begin{equation*}
        \mathcal{F}_{v_{\alpha}}^{\lambda}H^0(Y, \mathscr{O}_Y(D)) = H^0\left( Y, \mathcal{F}_{v_{\alpha}}^{\lambda}(\mathscr{O}_Y(D)) \right). 
    \end{equation*}
    By Lemma \ref{valuative_ideal_of_qm}, these filtrations are given by 
    \begin{equation*}
        \mathcal{F}_{v_{\alpha}}^{\lambda}(\mathscr{O}_Y(D)) = \sum_{\langle \alpha, m \rangle \geq \lambda} \mathscr{O}_Y(D - E(m)). 
    \end{equation*}
    Thus we have $\mathcal{F}_{\alpha}^{\lambda}R_k \subset \mathcal{F}_{v_{\alpha}}^{\lambda}R_k$ and $\mathcal{F}_{\alpha}^{\lambda}M_k \subset \mathcal{F}_{v_{\alpha}}^{\lambda}M_k$, and for each fixed $\lambda$, the equalities hold for $k \gg 0$ since $H$ is ample. Let $s \in M_0$, with $v_{\alpha}(s) = \lambda$. Then $\mathrm{ord}_{\alpha}(s) \leq \lambda$. Choose $k \gg 0$, depending on $\lambda$, such that $\mathcal{F}_{\alpha}^{\lambda}M_k = \mathcal{F}_{v_{\alpha}}^{\lambda}M_k$. Furthermore, we may assume $kH$ is Cartier and there exists $f \in R_k$ not vanishing at $\zeta$, since $H$ is ample. Then it is clear that $v_{\alpha}(f) = 0 = \mathrm{ord}_{\alpha}(f)$. Then 
    \begin{equation*}
        v_{\alpha}(fs) = v_{\alpha}(f) + v_{\alpha}(s) = \lambda, 
    \end{equation*}
    hence $fs \in \mathcal{F}_{v_{\alpha}}^{\lambda}M_k = \mathcal{F}_{\alpha}^{\lambda}M_k$. That is, $\mathrm{ord}_{\alpha}(fs) \geq \lambda$. Since $\mathrm{ord}_{\alpha}$ is multiplicative, we have 
    \begin{equation*}
        \mathrm{ord}_{\alpha}(s) = \mathrm{ord}_{\alpha}(fs) - \mathrm{ord}_{\alpha}(f) \geq \lambda. 
    \end{equation*}
    Therefore we have $\mathrm{ord}_{\alpha}(s) = \lambda = v_{\alpha}(s)$. 
\end{proof}

\begin{thm} \label{finite_generation_thm}
    Let $(Y, \Delta)$ be an lc pair, $E_1, \ldots, E_r$ be reduced divisors on $Y$ with $E = \sum_{i = 1}^{r} E_i = \lfloor \Delta \rfloor$, and $g \colon Y \to S$ be a proper morphism such that $-(K_Y + \Delta)$ is $g$-ample. Assume every lc center of $(Y, \Delta)$ is contained in $\mathrm{Supp}(E)$, and each $E_i$ is $\mathbb{Q}$-Cartier at the lc centers. For each $m = (m_1, \ldots, m_r) \in \mathbb{Z}^r$, write 
    \begin{equation*}
        E(m) \coloneqq \sum_{i=1}^{r} \max(m_i, 0) E_i. 
    \end{equation*}
    Let $L_1, \ldots, L_s$ be $\mathbb{Q}$-Cartier $\mathbb{Z}$-divisors on $Y$, 
    \begin{equation*}
        R = R(L_1, \ldots, L_s) = \bigoplus_{n = (n_1, \ldots, n_s) \in \mathbb{N}^s}H^0(Y, \mathscr{O}_Y(n_1L_1 + \cdots + n_sL_s))
    \end{equation*}
    be the multi-section ring, and 
    \begin{equation*}
        \mathcal{R} = \bigoplus_{n \in \mathbb{N}^s} \bigoplus_{m \in \mathbb{Z}^r} H^0(Y, \mathscr{O}_Y(n_1L_1 + \cdots + n_sL_s - E(m))) t_1^{-m_1} \cdots t_r^{-m_r} \subset R[t_1^{\pm 1}, \ldots, t_r^{\pm r}]
    \end{equation*}
    be the extended Rees algebra. Assume $Z = \bigcap_{i=1}^{r} E_i$ is irreducible with the generic point $\zeta \in Z$. Suppose $\alpha = (\alpha_1, \ldots, \alpha_r)\in \mathbb{R}_{>0}^r$ and $v_{\alpha} \in \mathrm{QM}^{\circ}_{\zeta}(Y, E)$ is the corresponding quasi-monomial valuation. Let $\mathcal{F}_{v_{\alpha}}$ be the filtration on $R$ induced by $v_{\alpha}$, with the associated graded algebra $\mathrm{gr}_{v_{\alpha}}(R)$. Then there is a canonical isomorphism 
    \begin{equation*}
        \mathrm{gr}_{v_{\alpha}}(R) \simeq \mathcal{R}/(t_1, \ldots, t_r)\mathcal{R}. 
    \end{equation*}
    In particular, $\mathrm{gr}_{v_{\alpha}}(R)$ is of finite type over $A = H^0(S, \mathscr{O}_S)$. 
\end{thm}

\begin{proof}
    Apply Lemma \ref{adjust_coeff_of_Fano_type_model} for $\phi = \mathrm{id}$, $\Delta_0 = E$, and $\Delta_1 = 0$, we conclude that $Y$ is of klt type. Then there is a small birational modification $\pi \colon Y' \to Y$ such that $E_i' = \pi^{-1}_{*}E_i$ is $\mathbb{Q}$-Cartier for each $i$, and $\pi$ is a local isomorphism over the locus where every $E_i$ is $\mathbb{Q}$-Cartier (see \cite[Lem.\ 4.7]{Zhuang_boundednessI}). Apply Lemma \ref{adjust_coeff_of_Fano_type_model} for $\phi = \pi^{-1}$ and $\Delta_0 = \Delta_1 = 0$, we get a pair $(Y', \Delta')$ satisfying the same conditions as in the Theorem. By Lemma \ref{compare_sections}, we may replace $(Y, \Delta)$ with $(Y', \Delta')$, hence assume every $E_i$ is $\mathbb{Q}$-Cartier. 

    Suppose $L$ is a $\mathbb{Q}$-Cartier $\mathbb{Z}$-divisor on $Y$, we claim that 
    \begin{equation*}
        \mathcal{F}_{v_{\alpha}}^{\lambda} H^0(Y, \mathscr{O}_Y(L)) = \sum_{\langle \alpha, m \rangle \geq \lambda} H^0(Y, \mathscr{O}_Y(L - E(m))). 
    \end{equation*}
    To show this, let $H$ be an $g$-ample divisor on $Y$, and consider the extended Rees algebra $\mathcal{R}_{L,H}$ for the multi-section ring 
    \begin{equation*}
        R(L, H) = \bigoplus_{\ell,k \geq 0} H^0(Y, \mathscr{O}_Y(\ell L + kH)). 
    \end{equation*}
    Then $\mathcal{R}_{L,H}$ is flat over $\Bbbk[t_1, \ldots, t_r]$, and $\mathcal{R}_{L,H}/(t_1, \ldots, t_r)\mathcal{R}_{L,H}$ is an integral domain, by Theorem \ref{the_central_fiber_of_Rees_alg}. For each $\ell \geq 0$, let $\mathcal{M}_{\ell}$ be the extended Rees module for $\bigoplus_{k \geq 0} H^0(Y, \mathscr{O}_Y(\ell L + kH))$, so that $\mathcal{M}_0 = \mathcal{R}_H$ is the extended Rees algebra for the section ring of $H$, and 
    \begin{equation*}
        \mathcal{R}_{L,H}/(t_1, \ldots, t_r)\mathcal{R}_{L,H} = \bigoplus_{\ell \geq 0} \mathcal{M}_{\ell}/(t_1, \ldots, t_r)\mathcal{M}_{\ell}. 
    \end{equation*}
    Then $\mathcal{M}_{1}/(t_1, \ldots, t_r)\mathcal{M}_{1}$ is a torsion-free module over the integral domain $\mathcal{M}_{0}/(t_1, \ldots, t_r)\mathcal{M}_{0}$. Thus, by Lemma \ref{integral_central_fiber_is_valuative}, we get the desired formula for $\mathcal{F}_{v_{\alpha}}^{\lambda} H^0(Y, \mathscr{O}_Y(L))$. 

    Therefore, $\mathrm{gr}_{v_{\alpha}}(R) \simeq \mathcal{R}/(t_1, \ldots, t_r)\mathcal{R}$ by Lemma \ref{graded_of_compatible_filtration}. It is of finite type since $\mathcal{R}$ is of finite type. 
\end{proof}

\begin{rem} \label{get_rid_of_ample}
    Note that we need an auxiliary ample divisor $H$ in the proof to apply Lemma \ref{integral_central_fiber_is_valuative}, and get the result for the multi-section ring of arbitrary divisors. Thus, even if one is only interested in the section ring of one specific divisor (for example, the pullback of an anti-canonical divisor from a Fano variety), it is more convenient to consider multi-section rings throughout the proof. 
\end{rem}

By examine the whole proof, we can get more information on the associated graded algebra. 

\begin{cor} \label{graded_ring_from_polyhedral_divisorial_sheaf}
    Keep the assumptions in Theorem \ref{the_central_fiber_of_Rees_alg}. Assume that $Z = \bigcap_{i=1}^{r} E_i$ is irreducible with the generic point $\zeta \in Z$. Write 
    \begin{equation*}
        \overline{\mathcal{R}} = \mathcal{R}/(t_1, \ldots, t_r)\mathcal{R} = \bigoplus_{u \in \mathbb{N}^{s+r}} \overline{\mathcal{R}}_u, 
    \end{equation*}
    with the grading by $\mathbb{N}^{s+r}$ induced by the natural grading on $\mathcal{R}$. Let $\sigma \subset \mathbb{R}^{s+r}$ denote the closed convex cone spanned by all $u \in \mathbb{N}^{s+r}$ such that $\overline{\mathcal{R}}_u \neq 0$. Then $\sigma$ is a convex rational polyhedral cone, and for $u = (n,m) \in \mathbb{N}^{s+r}$, $u \in \sigma$ if and only if $\zeta \notin \mathbf{B}(n_1L_1 + \cdots + n_sL_s - m_1E_1 - \cdots - m_rE_r)$. 

    Assume furthermore that $(Y, \Delta)$ is dlt. Then there exists a positive integer $\ell$, a projective birational morphism $\rho \colon W \to Z$ with $W$ regular, and a collection of invertible $\mathscr{O}_W$-modules $\{\mathscr{L}(u)\}_{u \in \sigma \cap \ell \mathbb{N}^{s+r}}$ with nonzero maps 
    \begin{equation*}
        \mu_{u,u'} \colon \mathscr{L}(u) \otimes \mathscr{L}(u') \to \mathscr{L}(u+u')
    \end{equation*}
    such that the following hold: 
    \begin{enumerate}[label=\emph{(\arabic*)}, nosep]
        \item $\bigoplus_{u} \mathscr{L}(u)$ form a commutative $\mathscr{O}_W$-algebra, 
        \item there is a canonical isomorphism of graded rings 
        \begin{equation*}
            \overline{\mathcal{R}}_{(\ell)} \simeq \bigoplus_{u \in \sigma \cap \ell \mathbb{N}^{s+r}} H^0(W, \mathscr{L}(u))
        \end{equation*}
        where $\overline{\mathcal{R}}_{(\ell)} = \bigoplus_{u \in \ell\mathbb{N}^{s+r}} \overline{\mathcal{R}}_u \subset \overline{\mathcal{R}}$ is the $\ell$-th Veronese subring. 
        \item there is a decomposition $\sigma = \bigcup_{\lambda \in \Lambda} \sigma_{\lambda}$ of $\sigma$ into finitely many convex rational polyhedral cones $\sigma_{\lambda}$, such that if $u, u' \in \sigma_{\lambda} \cap \ell \mathbb{N}^{s+r}$ for some $\lambda \in \Lambda$, then the map $\mu_{u,u'}$ is an isomorphism. 
    \end{enumerate}
\end{cor}

\begin{proof}
    As in Theorem \ref{the_central_fiber_of_Rees_alg}, we may assume that each $E_i$ is $\mathbb{Q}$-Cartier. 
    
    Let $u = (n,m) \in \mathbb{N}^{s+r}$. In the proof of Theorem \ref{the_central_fiber_of_Rees_alg}, we get 
    \begin{equation*}
        \overline{\mathcal{R}}_u \simeq H^0(Y, \mathscr{O}_{Y}(L(u))) \bigg/ \sum_{i=1}^r H^0(Y, \mathscr{O}_{Y}(L(u) - E_i)). 
    \end{equation*}
    where $L(u) = n_1L_1 + \cdots + n_sL_s - m_1E_1 - \cdots - m_rE_r$. If $\zeta \in \mathbf{B}(L(u))$, then $\overline{\mathcal{R}}_{u} = 0$ by Lemma \ref{Koszul_exact_on_sections}(2). Conversely, assume that $\zeta \notin \mathbf{B}(L(u))$. Then there is $c \in \mathbb{Z}_{>0}$ such that $L(cu) = cL(u)$ is Cartier, and $\zeta$ is not in the base locus of $L(cu)$. That is, there is a section $f \in H^0(Y, \mathscr{O}_Y(L(cu)))$ that does not vanish at $\zeta$. Then the image of $f$ in $\overline{\mathcal{R}}_{cu}$ is nonzero by Lemma \ref{Koszul_exact_on_sections}(2). Hence $cu \in \sigma$, so $u \in \sigma$ as well. 
    
    In the following, assume $(Y, \Delta)$ is dlt. Let $u \in \sigma \cap \mathbb{N}^{s + r}$. The proof of Lemma \ref{Koszul_exact_on_sections} says that there is a birational contraction 
    \begin{equation*}
        \begin{tikzcd}
            Y \ar[rr, "\phi_u", dashed] \ar[rd, "g"'] & & Y'_u \ar[ld, "g'_u"] \\
            & S &
        \end{tikzcd}
    \end{equation*}
    such that $\phi_u$ is a local isomorphism at $\zeta$, and $\overline{\mathcal{R}}_u \simeq H^0(Y'_u, \mathscr{L}'_u(u))$ with 
    \begin{equation*}
        \mathscr{L}'_u(u) = \mathscr{O}_{Y'_u}(L'_u(u)) \bigg/ \sum_{i=1}^r \mathscr{O}_{Y'_u}(L'_u(u) - E'_{i,u}). 
    \end{equation*}
    where $L'_u(u) = (\phi_u)_{*}L(u)$ and $E'_{i,u} = (\phi_u)_{*}E_i$. This birational contraction is given by Lemma \ref{pass_to_sa_model}, where $Y'_u$ is a small birational modification of a semi-ample model for $L(u)$ over $S$ given by an MMP. Thus we may assume the following hold: There is a decomposition $\sigma = \bigcup_{\lambda \in \Lambda} \sigma_{\lambda}$ of $\sigma$ into finitely many convex rational polyhedral cones $\sigma_{\lambda}$, such that if $u, w \in \sigma_{\lambda} \cap \mathbb{N}^{s+r}$ for some $\lambda \in \Lambda$, then $Y_u' = Y_w'$ and $\phi_u = \phi_{w}$ (see \cite[Cor.\ 1.1.5]{BCHM}). In particular, there are only finitely many distinct $Y'_u$. 
    
    Choose $\ell \in \mathbb{Z}_{>0}$ such that $\ell L'_u(u) = L'_{\ell u}(\ell u)$ is Cartier on $Y'_u = Y'_{\ell u}$ for all $u \in \sigma \cap \mathbb{N}^{s+r}$. Thus, for all $u \in \sigma \cap \ell \mathbb{N}^{s+r}$, $L'_u(u)$ is Cartier, and
    \begin{equation*}
        \mathscr{L}'_u(u) \simeq \mathscr{O}_{Y'_u}(L'_u(u))|_{Z'_u}, 
    \end{equation*}
    where $Z'_u = \bigcap_{i=1}^{r} E'_{i,u} \subset Y'_u$. By Lemma \ref{adjust_coeff_of_Fano_type_model}, $(Y'_u, \Delta'_u)$ is dlt for some effective $\mathbb{Q}$-divisor $\Delta'_u$ on $Y'_u$ with $\lfloor \Delta'_u \rfloor = E'_u = \sum_{i=1}^{r} E'_{i,u}$. Thus $Z'_u$ is normal, and $(Y'_u, E'_u)$ is snc at the generic point $\zeta'_u = \phi_u(\zeta)$ of $Z'_u$. Then there exists a log resolution $\pi \colon V \to (Y, E)$, such that $\pi$ is a local isomorphism over $\zeta$, and 
    \begin{equation*}
        \pi_u = \phi_u \circ \pi \colon V \to (Y'_u, E'_u)
    \end{equation*}
    is also a morphism and a log resolution for all $u \in \sigma \cap \mathbb{N}^{s+r}$. Let $W \subset V$ be the closure of $\pi^{-1}(\zeta)$. Then $W$ is a stratum of the snc pair $(V, \pi^{-1}_{*}E)$, so $W$ is regular. Moreover, $\pi$ induces a projective birational morphism $\rho \colon W \to Z$, and each $\pi_u$ induces a projective birational morphism $\rho_u \colon W \to Z'_u$. 

    Let $u \in \sigma \cap \ell \mathbb{N}^{s+r}$. Define 
    \begin{equation*}
        \mathscr{L}(u) \coloneqq \mathscr{O}_V(\pi_u^{*}L'_u(u))|_{W} \simeq \rho_u^{*}(\mathscr{O}_{Y'_u}(L'_u(u))|_{Z'_u}) \simeq \rho_u^{*}\mathscr{L}'_{u}(u). 
    \end{equation*}
    Then 
    \begin{equation*}
        \overline{\mathcal{R}}_u \simeq H^0(Z'_u, \mathscr{L}'_{u}(u)) \simeq H^0(W, \mathscr{L}(u)). 
    \end{equation*}
    For $u, v \in \sigma \cap \ell \mathbb{N}^{s+r}$, we claim that 
    \begin{equation*}
        \pi_u^{*}L'_u(u) + \pi_w^{*}L'_v(v) \leq \pi_{u+v}^{*}L'_{u+v}(u+v). 
    \end{equation*}
    If $u, v \in \sigma_{\lambda}$ for some $\lambda$, then clearly the equality holds. In general, $\pi_u^{*}L'_u(u) + \pi_v^{*}L'_v(v)$ is semi-ample over $S$, and 
    \begin{equation*}
        \pi_u^{*}L'_u(u) + \pi_v^{*}L'_v(v) \leq \pi^{*}L(u) + \pi^{*}L(v) = \pi^{*}L(u+v)
    \end{equation*}
    since $\phi_u$ is $L(u)$-non-positive and $\phi_v$ is $L(v)$-non-positive. Note that $\pi_{u+v} \colon W \to Y'_{u+v}$ is a semi-ample model for $\pi^{*}L(u+v)$ over $S$, and $L'_{u+v}(u+v) = (\pi_{u+v})_{*}\pi^{*}L(u+v)$. Hence 
    \begin{equation*}
        \pi_u^{*}L'_u(u) + \pi_v^{*}L'_v(v) \leq \pi_{u+v}^{*}L'_{u+v}(u+v). 
    \end{equation*}
    Thus, we have a commutative $\mathscr{O}_V$-algebra $\bigoplus_{u} \mathscr{O}_V(\pi_u^{*}L'_u(u))$, whose restriction to $W$ is a commutative $\mathscr{O}_W$-algebra $\bigoplus_{u} \mathscr{L}(u)$. The multiplication map $\mu_{u,v} \colon \mathscr{L}(u) \otimes \mathscr{L}(v) \to \mathscr{L}(u+v)$ is an isomorphism at the generic point of $W$, so $\mu_{u,v}$ is nonzero. Also, $\mu_{u,v}$ is an isomorphism if $u, v \in \sigma_{\lambda}$ for some $\lambda \in \Lambda$ in the decomposition $\sigma = \bigcup_{\lambda \in \Lambda} \sigma_{\lambda}$. 

    It remains to prove 
    \begin{equation*}
        \overline{\mathcal{R}}_{(\ell)} = \bigoplus_{u \in \sigma \cap \ell\mathbb{N}^{s+r}} \overline{\mathcal{R}}_u \simeq \bigoplus_{u \in \sigma \cap \ell\mathbb{N}^{s+r}} H^0(W, \mathscr{L}(u))
    \end{equation*}
    is an isomorphism of rings. This graded isomorphism fits into a commutative diagram
    \begin{equation*}
        \begin{tikzcd}
            \bigoplus_{u} \overline{\mathcal{R}}_u \ar[r, "\simeq"] \ar[d, hook] & \bigoplus_{u} H^0(W, \mathscr{L}(u)) \ar[d, hook] \\
            \bigoplus_{u} \mathscr{O}_Y(L(u)) \otimes K(Z) \ar[r, "\simeq"] & \bigoplus_{u} \mathscr{O}_{V}(\pi_u^{*}L'_u(u)) \otimes K(W)
        \end{tikzcd}
    \end{equation*}
    where the sums are taken over $u \in \sigma \cap \ell\mathbb{N}^{s+r}$. The vertical maps are injective ring homomorphisms by definition, and the lower horizontal map is an isomorphism of rings since $\pi \colon V \to Y$ and all $\pi_u \colon V \to Y'_u$ are local isomorphisms at the generic point of $W \subset V$. It follows that the upper horizontal map is also an isomorphism of rings. 
\end{proof}

\section{An example} In this section, we give an example demonstrating that Theorem \ref{finite_generation_thm} applies to more quasi-monomial valuations than those monomial lc places of special $\mathbb{Q}$-complements as in \cite{LXZ_finite_generation} and \cite{XZ_stable_deg}. 

\begin{exmp} \label{example_weakly_special}
    Let $\Bbbk$ be a field of characteristic $0$. Let $X \subset \mathbb{P}^n$ be a smooth hypersurface of degree $d$, and $E_1, E_2 \subset X$ be two smooth hyperplane sections. Assume that $d < n - 1$, so that $-(K_X + E_1 + E_2)$ is ample. The scheme-theoretic intersection $Z = E_1 \cap E_2$ is a hypersurface of degree $d$ in $\mathbb{P}^{n-2}$. Note that $(X, E_1 + E_2)$ is snc at the smooth locus of $Z$. Thus, if $Z$ is integral with the generic point $\zeta$, then we can consider quasi-monomial valuations in $\mathrm{QM}_{\zeta}^{\circ}(X, E_1 + E_2) \simeq \mathbb{R}_{> 0}^2$. 

    By inversion of adjunction, if $Z$ is klt, then $(X, E_1 + E_2)$ is dlt, hence any $v \in \mathrm{QM}^{\circ}_{\zeta}(X, E_1 + E_2)$ is a monomial lc place of special $\mathbb{Q}$-complement on $X$. If $Z$ is lc, then $(X, E_1 + E_2)$ is lc, hence Theorem \ref{main_thm_1} applies to any $v \in \mathrm{QM}^{\circ}_{\zeta}(X, E_1 + E_2)$. 
    
    Let $R = \bigoplus_{n \in \mathbb{N}} H^0(X, \mathscr{O}_X(n))$ be the homogeneous coordinate ring of $X$. We can compute the graded algebra $\mathrm{gr}_v(R)$ for $v \in \mathrm{QM}^{\circ}_{\zeta}(X, E_1 + E_2)$ as follows. Let $P = \Bbbk[x_1, \ldots, x_{n+1}]$, so that $R = P/(f)$ where $f$ is a homogeneous polynomial of degree $d$ defining $X \subset \mathbb{P}^n$. We may assume that $E_1$ and $E_2$ are defined by $x_1 = 0$ and $x_2 = 0$, respectively, hence $Z$ is defined by the image $\bar{f}$ of $f$ in $\Bbbk[x_3, \ldots, x_{n+1}]$. We can also write 
    \begin{equation*}
        f = f_0 + x_1f_1 + x_2f_2, 
    \end{equation*}
    where $f_0$ is the same polynomial as $\bar{f}$, but viewed as an element in $P$, and $f_1, f_2 \in P$. The vanishing orders along $x_1$ and $x_2$ induce filtrations $\mathcal{F}_{1}$ and $\mathcal{F}_{2}$ on $P$, respectively, which induce filtrations $\mathcal{F}_{E_1}$ and $\mathcal{F}_{E_2}$ on $R = P/(f)$. Then $\mathrm{gr}_{\mathcal{F}_1}(\mathrm{gr}_{\mathcal{F}_2}(P)) \simeq \Bbbk[x_1, \ldots, x_{n+1}]$, and the initial term of $f$ in $\mathrm{gr}_{\mathcal{F}_1}(\mathrm{gr}_{\mathcal{F}_2}(P))$ is $f_0$. Hence 
    \begin{equation*}
        \mathrm{gr}_{\mathcal{F}_{E_1}}(\mathrm{gr}_{\mathcal{F}_{E_2}}(R)) \simeq \Bbbk[x_1, \ldots, x_{n+1}]/(f_0). 
    \end{equation*}
    By Theorem \ref{finite_generation_thm}, $\mathrm{gr}_v(R) \simeq \mathrm{gr}_{\mathcal{F}_{E_1}}(\mathrm{gr}_{\mathcal{F}_{E_2}}(R)) \simeq \Bbbk[x_1, \ldots, x_{n+1}]/(f_0)$ for any $v \in \mathrm{QM}^{\circ}_{\zeta}(X, E_1 + E_2)$. Note that the $\mathbb{N}$-grading on $\mathrm{gr}_v(R)$ induced by the $\mathbb{N}$-grading on $R$ is the usual one, with $\deg(x_i) = 1$ for all $i$. Geometrically, $v$ induces a degeneration of $X$ to $X_0 = \mathop{\mathrm{Proj}}(\mathrm{gr}_v(R)) \subset \mathbb{P}^n$, which is the double projective cone over $Z \subset \mathbb{P}^{n-2}$. If $Z$ is lc but not klt, then $X_0$ is also lc but not klt (see \cite[Lem.\ 3.1]{Kollar_singularity}). Thus, by \cite[Thm.\ 4.2]{LXZ_finite_generation}, $v$ is not a monomial lc place of a special $\mathbb{Q}$-complement in this case. 

    To get a very concrete example, consider the case $d = 3$ and $n = 5$. Let $C \subset \mathbb{P}^2 = \mathop{\mathrm{Proj}}(\Bbbk[x, y, z])$ be a smooth cubic curve, defined by an equation $f(x, y, z) = 0$. For example, $f(x, y, z) = x^3 + y^3 + yz^2$. Let $X \subset \mathbb{P}^5 = \mathop{\mathrm{Proj}}(\Bbbk[u, v, w, x, y, z])$ be the cubic hypersurface defined by the equation
    \begin{equation*}
        u^3 - v^3 + (u + v)w^2 + f(x, y, z) = 0. 
    \end{equation*}
    Let $E_1$ and $E_2$ be the hyperplane sections of $X$ defined by $u = 0$ and $v = 0$, respectively. Then $E_1$ and $E_2$ are smooth, and $Z = E_1 \cap E_2 \subset \mathbb{P}^3 = \mathop{\mathrm{Proj}}(\Bbbk[w, x, y, z])$ is defined by $f(x, y, z) = 0$. That is, $Z$ is the projective cone over $C$. Thus $Z$ is lc but not klt. Now, a valuation $v \in \mathrm{QM}^{\circ}_{\zeta}(X, E_1 + E_2)$ gives 
    \begin{equation*}
        \mathrm{gr}_v(R) \simeq \Bbbk[u, v, w, x, y, z]/(f(x, y, z)), 
    \end{equation*}
    and induces a degeneration of $X$ to the triple projective cone $X_0$ over $C$. 
\end{exmp}

\printbibliography

\end{document}